\documentclass[a4paper,10pt]{amsart}

\usepackage{amsthm}
\usepackage{amssymb}
\usepackage{amsmath}
\usepackage{mathtools}
\usepackage{pdfpages}
\usepackage{changes}
\usepackage{exscale,cite,color,amsopn}
\usepackage[colorlinks=true,pdftex,unicode=true,linktocpage,bookmarksopen,hypertexnames=true]{hyperref}
\usepackage{aliascnt}
\usepackage{colortbl}
\usepackage{enumitem}
\usepackage{verbatim}
\usepackage{hyperref}
\usepackage{tikz-cd}
\usepackage[capitalise]{cleveref}

\usetikzlibrary{trees}

\renewcommand{\leq}{\leqslant}
\renewcommand{\geq}{\geqslant}
\DeclareMathOperator{\Fix}{Fix}
\DeclareMathOperator{\id}{id}

\DeclareMathOperator{\Ret}{Ret}
\DeclareMathOperator{\Soc}{Soc}
\DeclareMathOperator{\Sym}{Sym}
\DeclareMathOperator{\Sq}{Sq}

\DeclareMathOperator{\ord}{ord}

\newcommand{\N}{\mathbb{N}}
\newcommand{\Z}{\mathbb{Z}}

\newcommand{\Aut}{\operatorname{Aut}}

\newcommand{\End}{\mathrm{End}}

\newcommand{\G}{\mathcal{G}}

\newcommand{\Gmult}{\mathcal{G}^{\circ}}
\newcommand{\Gadd}{\mathcal{G}^+}
\newcommand{\Gbr}{\mathcal{G}}

\newcommand{\Pri}[1]{#1^{\prime}}

\makeatletter
\numberwithin{equation}{section}
\numberwithin{figure}{section}
\numberwithin{table}{section}
\newtheorem{thm}{Theorem}[section]
\newtheorem*{thm*}{Theorem}
\newtheorem{lem}[thm]{Lemma}
\newtheorem{cor}[thm]{Corollary}
\newtheorem{pro}[thm]{Proposition}

\theoremstyle{definition}

\newtheorem*{convention*}{Convention}
\newtheorem{rem}[thm]{Remark}

\makeatother

\title[Indecomposable involutive solutions of size $p^2$]{Indecomposable involutive set-theoretical solutions to the Yang--Baxter equation of size $p^2$}
\author{Carsten Dietzel, Silvia Properzi, Senne Trappeniers}

\address[Carsten Dietzel]{Department of Mathematics and Data Science, Vrije Universiteit Brussel, Pleinlaan 2, 1050 Brussel, Belgium}
\email{Carsten.Dietzel@vub.be}

\address[Silvia Properzi]{Department of Mathematics and Data Science, Vrije Universiteit Brussel, Pleinlaan 2, 1050 Brussel, Belgium}
\email{Silvia.Properzi@vub.be}

\address[Senne Trappeniers]{Department of Mathematics and Data Science, Vrije Universiteit Brussel, Pleinlaan 2, 1050 Brussel, Belgium}
\email{Senne.Trappeniers@vub.be}

\begin{document}
\begin{abstract}
    This article focuses on indecomposable involutive non-degenerate set-theoretical solutions to the Yang--Baxter equation. More specifically, we give a full classification of those solutions which are of size $p^2$, for $p$ a prime. We do this through a thorough analysis of their associated permutation braces and using the language of cycle sets.
\end{abstract}
\maketitle

\noindent \textbf{Keywords:} \textit{brace, cycle set, Yang--Baxter equation}

\section*{Introduction}

The \emph{quantum Yang--Baxter equation}, named after the physicists Chen-Ning Yang and Rodney James Baxter, first appeared in the context of statistical mechanics and integrable systems \cite{Baxter_YB,Yang_YB}. Given a (typically complex) vector space $V$, call a linear operator $R: V^{\otimes 2} \to V^{\otimes 2}$ a \emph{solution} to the \emph{Yang--Baxter equation} if $R$ is bijective and satisfies the following equation within $\End(V^{\otimes 3})$:
\begin{equation}
    (R \otimes \id_{V}) (\id_{V} \otimes R) (R \otimes \id_{V}) =   (\id_{V} \otimes R) (R \otimes \id_{V}) (\id_{V} \otimes R).  \tag{YBE}
\end{equation}
Outside of physics, solutions to this equation have been recognized to be of importance in low-dimensional topology, as results of Turaev \cite{Turaev_links} show that each Yang--Baxter operator gives rise to an associated knot invariant. While these operators are therefore of high importance in knot theory, they are very difficult to construct; initially, the only source of solutions was the theory of quantum groups. However, Drinfeld \cite{Drinfeld_Problems} observed that a particular class of solutions provides a handier way to obtain classes of solutions by demanding that $R$ permutes the tensors $v_i \otimes v_j$ where $(v_i)_{i\in I}$ is a basis of $V$. By doing so, one gets a discretized version of the Yang--Baxter equation, the \emph{set-theoretical Yang--Baxter equation}, which consists of a set $X$ and a bijection $r: X^2 \to X^2$ that satisfies the following equation in $\Sym_{X^3}$, the symmetric group on the set $X^3$:

\begin{equation}
  (r \times \id_X) (\id_X \times r)(r \times \id_X) = (\id_X \times r)(r \times \id_X) (\id_X \times r). \tag{SYBE}
\end{equation}

Besides the simpler, combinatorial flavour of this equation, Drinfeld's intention was to submit these solutions to a deformation process from which novel linear solutions can be obtained.

Indeed, set-theoretical solutions admit an analysis via algebraic methods, the first steps being made by Etingof, Schedler and Soloviev \cite{ESS_YangBaxter} and Gateva-Ivanova and Van den Bergh \cite{GIVdB_IType}.

A \emph{homomorphism} of solutions $f:(X,r)\to (Y,s)$ is a map $f: X \to Y$ such that $(f \times f)r = s(f \times f)$. An \emph{isomorphism} of solutions is a bijective homomorphism.

Writing a set-theoretical solution $(X,r)$ as $r(x,y) = (\lambda_x(y), \rho_y(x))$ we call it
\begin{enumerate}
    \item \emph{non-degenerate}, if the mappings $\lambda_x, \rho_x$ are bijective for all $x \in X$,
    \item \emph{involutive}, if $r^2 = \id_{X^2}$.
\end{enumerate}
By the size of a set-theoretical solution, we mean the size of $X$.
\begin{convention*}
    From now on, we simply write \emph{set-theoretical solutions} or \emph{solutions} when talking about set-theoretical solutions to the Yang--Baxter equation that are finite, involutive and non-degenerate.
\end{convention*}

Rump has shown in \cite{Rump_Decomposition} that the binary operation $x \cdot y = \lambda_x^{-1}(y)$ turns a solution $X$ into a particularly easy algebraic structure; a \emph{(non-degenerate) cycle set}, which is defined as a set $X$ with a binary operation $(x,y) \mapsto x \cdot y$ such that the following axioms are satisfied:

\begin{align}
    (x \cdot y) \cdot (x \cdot z) & = (y \cdot x) \cdot (y \cdot z) \quad \forall x,y,z \in X, \tag{C1} \label{eq:c1_cycloid_equation} \\
    \textnormal{the map} \quad y \mapsto \sigma_x(y) & = x \cdot y \ \textnormal{ is bijective for all } x \in X, \tag{C2} \label{eq:c2} \\
    \textnormal{the \emph{square map}} \quad x \mapsto \Sq(x) & = x \cdot x \ \textnormal{ is bijective.} \label{eq:c3} \tag{C3}
\end{align}

\begin{convention*}
    The cycle sets considered in this work will always be assumed to be of finite size and will be referred to as \emph{cycle sets}, without the predicate \emph{non-degenerate}.
\end{convention*}

Cycle sets are in bijective correspondence with set-theoretical solutions to the Yang--Baxter equation.
More precisely, given a cycle set $(X,\cdot)$ then $(X,r_X)$
is a set-theoretical solution to the YBE,
where $r_X:X^2\to X^2$ is given by 
$$r_X(x,y)=(\sigma_x^{-1}(y),\sigma_x^{-1}(y)\cdot x).$$

\emph{Homomorphisms} of cycle sets are defined in the obvious way. Consequently, this holds for all derived notions, such as \emph{isomorphisms}. Note that isomorphisms of cycle sets translate to isomorphisms of set-theoretical solutions under the correspondence described above. As usual, if there is no danger of confusion, a cycle set $(X,\cdot)$ is simply referred to by its underlying set $X$.

Of particular importance for understanding general cycle sets is the class of \emph{indecomposable} cycle sets. A cycle set $X$ is called \emph{indecomposable} if there are no subsets $X_1,X_2 \neq \emptyset$ closed under the cycle set operation such that $X=X_1\cup X_2$. In a sense, they are elementary building blocks for cycle sets and turn out to be significantly more well-behaved. This can be seen from computations of Akg\"un, Mereb and Vendramin \cite{AMV_Cyclesets} that imply that the number of cycle sets of given cardinality grows in an aggressive manner, while the number of indecomposable cycle sets stays quite small. This indicates that there are strong obstructions for a cycle set to be indecomposable. Note that a cycle set $(X,\cdot)$ is indecomposable if and only if the corresponding set-theoretical solution $(X,r)$ is indecomposable in the sense that there is no partition $X = A \sqcup B$ ($A,B \neq \emptyset$) such that $r(A^2) = A^2$ and $r(B^2) = B^2$.

Indeed, up to isomorphism, there is only one indecomposable cycle set of size $p$, $p$ a prime number, that is given by $X = \Z_p$ with the operation $x \cdot y = y+1$ \cite[Theorem 2.13]{ESS_YangBaxter}. Furthermore, Jedli\v{c}ka and Pilitowska \cite{jedlicka_pilitowska} have recently classified all indecomposable cycle sets of multipermutation level 2 by group-theoretic means. In particular, this includes all indecomposable cycle sets of order $pq$ that are of \emph{finite multipermutation level}, where $p,q$ are (not necessarily distinct) primes. Ced\'o and Okni\'nski \cite[Section 5]{CO_simple} have described a class of simple cycle sets of size $p^2$, $p$ a prime. In \cite{CO_SquarefreeIndecomposable}, the same authors have recently proven that all indecomposable cycle sets of squarefree size are of finite multipermutation level. 

In this article, we build on these results by providing in \cref{thm:final_classification_cycle_sets} a full classification of indecomposable cycle sets of size $p^2$, $p$ a prime. In \cref{cor:final_classification_solutions} we reformulate this classification in terms of set-theoretical solutions.

Note that the classification of indecomposable retractable solutions of size $p^2$ follows from the theoretical classification of indecomposable solutions of multipermutation level $2$ by Jedli\v{c}ka and Pilitowska in \cite{jedlicka_pilitowska}. In \cref{sec: finite mpl}, we merely give them in a more explicit form building further upon their results, which also gives a more uniform presentation compared to the irretractable case.

The strategy to tackle the irretractable case is 
to first restrict to the case when the permutation group of a solution is a $p$-group (\cref{sec:irretractable_cycle_sets_sylow}) and then construct all irretractable solutions by an extension procedure (\cref{sec: controlling p'}).

We conclude this article by comparing the found solutions with a family of solutions that have recently been constructed by Ced\'o and Okni\'nski \cite{CO_simple}, thus answering \cite[Question 7.4]{CO_simple} affirmatively. Furthermore, we give an enumeration of the irretractable solutions.

\section{Preliminaries}\label{sec:preliminaries}

Recall that a \emph{brace} is a triple $(A,+,\circ),$ where $(A,+)$ is an abelian group and $(A,\circ)$ is a group such that for all $a,b,c \in A$ a skew left distributivity holds $$a\circ(b+c)=a\circ b-a+a\circ c,$$ 
where $-a$ denotes the inverse of $a$ in $(A,+)$. 
Similarly, the inverse of an element $a\in A$ with respect to the operation $\circ$ will be denoted by $\overline{a}.$

The group $A^+=(A,+)$ is the \emph{additive group} of $A$ and $A^\circ=(A,\circ)$ is the 
\emph{multiplicative group} of $A$.
We will also denote the conjugation by $a$ in $A^\circ$ as ${}^ab=a\circ b\circ \overline{a}$, for all $a,b\in A$. For $n\in \Z$, $a\in A$, the $n$-th power of $a$ in $A^+$ is denoted $na$ and the $n$-th power in $A^\circ$ is denoted $a^{\circ n}$.

There is an action by automorphisms of $A^\circ$ on $A^+$ called the \emph{$\lambda$-action} defined by
\[
\lambda_a(b)=-a+a\circ b,
\]
for all $a,b\in A$. Related to this, there is also the operation 
$$a*b=\lambda_a(b)-b = -a + a \circ b - b,$$
which is easily verified to be left distributive with respect to the addition.
The set of fixed points of the $\lambda$-action is denoted by
\[
\Fix(A)=\{a\in A \colon b+a=b\circ a\text{ for all }b\in A\}=\{a\in A\mid b*a=0\text{ for all }b\in A\}.
\]

A \emph{subbrace} of $A$ is simply a subset $I$ of $A$ that is a subgroup of both $A^+$ and $A^\circ$.
 
A \emph{left ideal} of $A$ is a subset $I$ of $A$ such that $I$ is a subgroup of $A^+,$ and $\lambda_a(I)\subseteq I$, for all $a\in A$. Note that, in particular, every characteristic subgroup of $A^+$ is a left ideal. If additionally, $I$ is a normal subgroup of $A^\circ$, one calls $I$ an \emph{ideal} of $A$.

The \emph{socle of $A$} is the ideal 
\[
\Soc(A)=\ker\lambda=\{a\in A\mid a\circ b=a+b \text{ for all }b\in A\}.
\]

\begin{lem}
\label{lem:lambda_of_fix_is_conjugation}
Let $A$ be a brace, then 
\[
\lambda_a(b)=-a+{}^ab+\lambda_{{}^ab}(a)\textnormal{ for all }a,b\in A.
\]
In particular, if $a\in\Fix(A)$, then $\lambda_a(b)={}^ab$, for all $b\in A$ and thus $\lambda_a$ is a brace automorphism of $A$.
\end{lem}

\begin{proof}
    Observe that if $a,b\in A$, then
    \[
    \lambda_a(b) = -a + a\circ b = -a + {}^ab\circ a  = -a + {}^ab + \lambda_{{}^ab}(a).
    \qedhere
    \]
\end{proof}

\begin{lem} \label{lem:fixed_points_are_not_central}
Let $A$ be a brace. Then $\Fix(A) \cap Z(A^{\circ}) \subseteq \Soc(A)$. 
\end{lem}

\begin{proof}
Let $f \in \Fix(A) \cap Z(A^{\circ})$, then \cref{lem:lambda_of_fix_is_conjugation} implies that 
$\lambda_f(a) = {}^fa = a$ for all $a \in A$ , i.e. $f \in \Soc(A)$.
\end{proof}

\begin{thm}\label{thm:trivial_action_on_quotient}
Let $A$ be a brace and $B$ a subbrace such that $B^{\circ}$ acts trivially on $A^+ / B^+$ under the $\lambda$-action. Furthermore, assume that $B^{\circ}$ is normal in $A^{\circ}$. Then $B$ is an ideal of $A$.
\end{thm}
\begin{proof}
Let $b \in B$. We have to show that $\lambda_a(b) \in B$ for all $a \in A$. Note that ${}^ab \in B$ because $B$ is normal in $A^\circ$, therefore $\lambda_{{}^ab}(a) \in a + B$ as $B^{\circ}$ acts trivially on $A^+/B^+$. Using \cref{lem:lambda_of_fix_is_conjugation} and the normality of $B^+$, we deduce that
\[
\lambda_a(b) = -a + {}^ab + \lambda_{{}^ab}(a) \in -a + B + a + B = B.\qedhere
\]
\end{proof}
Given a prime $p$, we mean by a \emph{$p$-brace} a finite brace whose size is a power of $p$.

\begin{pro} \label{pro:maximal_subbraces_of_p_braces_are_ideals}
Let $A$ be a $p$-brace and $B \subseteq A$ a subbrace with $|A| = p|B|$. Then $B$ is an ideal of $A$.
\end{pro}

\begin{proof}
$A^+/B^+$ is cyclic of order $p$ and therefore is acted upon trivially by $B^{\circ}$. Furthermore, as $[A^{\circ}:B^{\circ}] = p$ is the smallest prime divisor of $|A|$, the subgroup $B^{\circ}$ is normal in $A^{\circ}$. It follows from \cref{thm:trivial_action_on_quotient} that $B$ is an ideal of $A$.
\end{proof}

\begin{lem} \label{lem:nontrivial_fix}
If $A$ is a $p$-brace, then $\Fix(A) \neq 0$. 
Moreover, if $A^\circ$ is abelian, then $\Soc(A)\neq0$.
\end{lem}

\begin{proof}
$\Fix(A)$ is the set of fixed points under the action $\lambda: A^{\circ} \to \Aut(A^+)$. It is well-known that an action of a $p$-group by automorphisms of a $p$-group always has non-trivial fixed points.

If we assume that $A^\circ$ is abelian,
\cref{lem:fixed_points_are_not_central} implies that $\Fix(A) \subseteq\Soc(A),$
hence $\Soc(A)\neq0$ as well.
\end{proof}

If $X$ is a set, we denote by $\Sym_X$ the symmetric group consisting of all permutations of $X$. If $n$ is a positive integer, we write $\Sym_n$ to indicate $\Sym_{\{ 1, \ldots, n\}}$.

Let $G$ be a group acting on a set $X$. A \emph{system of imprimitivity} is a partition $X = \bigsqcup_{i \in I}X_i$ that is invariant under the group action in the sense that for any $g \in G$, $i \in I$, there is a $j \in I$ such that $g \cdot X_i = X_j$. The subsets $X_i$ are called \emph{blocks} of the system. A system of imprimitivity is \emph{trivial} if either $|I|=1$ or $|X_i| = 1$ for all $i \in I$. Else it is non-trivial.

Recall that if $X$ is finite and $G$ is transitive, then $G$ acts transitively on the blocks of a system of imprimitivity, and $|X_i| = |X_j|$ for all $i,j \in X$. In particular, all $|X_i|$ divide $|X|$.

\begin{lem} \label{lem:lucchini}
    Let $G$ be a non-abelian $p$-group with a transitive action on a set $X$ of size $p^2$. Then $G$ has at most one non-trivial system of imprimitivity.
\end{lem}

\begin{proof}
If there is more than one non-trivial system of imprimitivity, a theorem of Lucchini \cite[Theorem 1]{Lucchini_imprimitive} implies that $G \leq \Sym_p \times \Sym_p$. But if $G$ is a $p$-group, this forces $G$ to be isomorphic to a subgroup of $\Z_p \times \Z_p$ and therefore to be abelian. 
\end{proof}

Let $G, H$ be groups such that $G$ acts on a set $X$. The \emph{wreath product} is the semidirect product $G \wr_X H = G \ltimes H^X$,
where $H^X=\{(h_x)_{x \in X}\mid h_x\in H\text{ for all }x\in X\}$ is the iterated direct product of $H$ with itself indexed by $X$
and $G$ acts on $H^X$ by $g \cdot (h_x)_{x \in X} = (h_{g^{-1} \cdot x})_{x \in X}$. If the action of $G$ on $X$ is clear, we will generally suppress the subscript-$X$ and write $G \wr H$.

If additionally, $H$ acts on a set $Y$ then $G \wr H$ acts on $X \times Y$ by permutations of the form
\[
(x,y) \mapsto (g \cdot x, h_x \cdot y).
\]
where $g \in G$, and $(h_x)_{x \in X}$ is a family of elements in $H$.

\begin{pro}\label{pro:impritive_solvable_action_on_Z2}
    Let $G\leq \Sym_{\Z_p\times \Z_p}$ be a transitive solvable group such that the sets $\{a\}\times \Z_p$, $a\in \Z_p$, form a system of imprimitivity, then $G$ is conjugated to a subgroup of $\mathrm{AGL}(1,p)\wr \mathrm{AGL}(1,p)$. If moreover $G$ is a $p$-group, then $G$ is conjugated to a subgroup of $\Z_p\wr \Z_p$.
\end{pro}
\begin{proof}
    Let $G$ be as above. As $G$ respects the given system of imprimitivity, it is well-known that $G\leq \Sym_p\wr \Sym_p$, where the action of the wreath product on $\Z_p\times \Z_p$ is precisely as described above.

    As a special case of \cite[Chapter II, Satz 3.2]{huppert1983endliche}, we know that every solvable transitive subgroup of $\Sym_{\Z_p}$ is conjugated to a subgroup of $\mathrm{AGL}(1,p)$. So we conclude that $G$ is conjugated to a subgroup of $\mathrm{AGL}(1,p)\wr \mathrm{AGL}(1,p)$. As $\Z_p\wr \Z_p$ is a $p$-Sylow subgroup of $\mathrm{AGL}(1,p)\wr \mathrm{AGL}(1,p)$ also the last part of the statement follows.
\end{proof}
Note that elements of $\Z_p \wr \Z_p$ are precisely permutations of the form
\[
(a,x) \mapsto (a + \beta, x + \gamma_a),
\]
with $\beta,\gamma_0,\ldots,\gamma_{p-1}\in \Z_p$.

On any brace $A$ we have a canonical cycle set structure $(A,\cdot)$ where $a\cdot b=\lambda_a^{-1}(b)$, for all $a,b\in A$.
A subset $X\subseteq A$ is a \emph{cycle base} of $A$ if it is a sub-cycle set of $(A,\cdot)$ and also $X$ generates the groups $A^+$ and $A^\circ$. A cycle base is called \emph{transitive} if $X$ forms an orbit under the $\lambda$-action of $A$, or equivalently if $X$ is an indecomposable cycle set. 

To a cycle set $(X,\cdot)$ we associate its \emph{structure group} $$(G(X),\circ)=\langle X\mid x\circ y =\sigma^{-1}_x(y)\circ (\sigma^{-1}_x(y)\cdot x)\text{ for all }x,y\in X\rangle.$$
There is a unique way of defining an addition on $G(X)$ such that $x+y=x\circ \sigma_x(y)$ for $x,y \in X$ and $(G(X),+,\circ)$ is a brace, the \emph{structure brace} of $X$. The canonical map $\iota: (X,\cdot) \to (G(X),\cdot):x\mapsto x$ is always an injective cycle set homomorphism, so we can identify $X$ with $\iota(X)$. Note that it follows immediately from the definition that $\iota(X)$ is a cycle base of $G(X)$.

To any cycle set $X$ we associate its \emph{permutation group}, defined as 
\begin{equation*}
    \Gbr(X)=\langle \sigma_x\mid x\in X\rangle \subseteq \Sym_X.
\end{equation*}

Recall that a cycle set $X$ is \emph{indecomposable} if and only if $\G(X)$ acts transitively on $X$. The above group structure will also be denoted by $\Gbr(X)^\circ$ and we obtain, as described in \cite{CJO_Braces}, a brace structure $(\Gbr(X),+,\circ)$ by defining $\sigma_x^{-1}+\sigma_y^{-1}=\sigma_x^{-1}\circ \sigma^{-1}_{\sigma_x(y)}$ for $x,y\in X$. Note that this means that we obtain a surjective brace homomorphism $G(X)\twoheadrightarrow \G(X):x\mapsto \sigma_x^{-1}$, where the kernel is precisely $\Soc(G(X))$.
As a result, the canonical map $$\pi:(X,\cdot)\mapsto (\Gbr(X),\cdot); \quad x\mapsto \sigma_x^{-1},$$
becomes a cycle set homomorphism. The image $\pi(X)$ is called the \emph{retraction} of $X$, denoted $\mathrm{Ret}(X)$. If $\pi$ is injective, hence $(X,\cdot)$ is isomorphic to a sub-cycle set of $(\Gbr(X),\cdot)$, then we say that $X$ is \emph{irretractable}. When $\pi$ is not injective, $(X,\cdot)$ is \emph{retractable}. Note that $\pi(x)=\pi(y)$ if and only if $\sigma_x=\sigma_y$. If moreover, we can obtain a one-element cycle set by iteratively repeating the retraction process on a cycle set $X$, we say that $X$ has \emph{finite multipermutation level}. More precisely, if $k$ is the smallest value such that $|\Ret^k(X)|=1$, then we say that $X$ has \emph{multipermutation level} $k$.

By a result of Ced\'o and Okni\'nski \cite[Lemma 3.3]{CO_SquarefreeIndecomposable} we have:
\begin{pro}\label{pro:fibres_have_same_cardinality}
    Let $f: X \twoheadrightarrow Y$ be a surjective cycle set homomorphism between finite cycle sets. If $X$ is indecomposable, then $Y$ is also indecomposable and the fibres $f^{-1}(y)$ all have the same cardinality. In particular, $|Y|$ divides $|X|$.
\end{pro}

\section{Indecomposable retractable cycle sets of size \texorpdfstring{$p^2$}{p2}}\label{sec: finite mpl}

Let $e_i$ be the canonical basis of $\bigoplus_{i\in \Z}\Z$. We define 
$$c_k=\begin{cases} 
\sum_{i=1}^{-k} -e_{1-i} & k<0\\
      0 & k=0 \\
      \sum_{i=1}^k e_i & k>0 
   \end{cases}$$

For the reader's convenience, we first recall a special case of the construction of indecomposable cycle sets of multipermutation level $2$ by Jedli\v{c}ka and Pilitowska \cite[Proposition 5.1, Proposition 5.7]{jedlicka_pilitowska}.
   
\begin{thm}\label{thm:jedlicka}
    Let $X$ be an indecomposable cycle set of multipermutation level 2 such that $|\Ret(X)|=m$, then there exist
    \begin{enumerate}
        \item a subgroup $H\leq \bigoplus_{i\in \Z}\Z$ such that $c_i-c_{i+m}\in H$, for all $i\in \Z$,
        \item $s\in \left( \bigoplus_{i\in \Z}\Z \right) /H$,
    \end{enumerate} such that $X$ is isomorphic to a cycle set of the form $X=(\Z\times\bigoplus_{i\in \Z}\Z)/{\sim}$ where 
    $$(a,x)\sim (b,y)\iff a-b\equiv 0 \pmod{m}\text{ and }x-y\equiv \frac{a-b}{m}s\pmod{H}$$
    and $$[(a,x)]\cdot [(b,y)]=[(b-1,y-c_{a-b}+c_{-b})].$$
    Moreover, different choices of $H$ and $s$ yield non-isomorphic cycle sets.
\end{thm}
We are now able to give a more explicit version of the above result. For an abelian group $(A,+)$ we define $\chi_0: A \to \Z$ as
\begin{align*}
            \chi_0(x) & = \begin{cases}
                1 & x = 0 \\
                0 & x \neq 0
            \end{cases}.
        \end{align*}
\begin{thm}\label{thm:jedlicka_improved}
    Let $X$ be an indecomposable cycle set of multipermutation level 2 such that $|\Ret(X)|=m$, then $X$ is isomorphic to a cycle set of the form $X=\Z_m\times A$ with the multiplication
    $$(a,x)\cdot (b,y)=(b+1,y+\chi_0(b)S+\Phi(b-a)),$$
    where $(A,+)$ is an abelian group and $\Phi:\Z_m\to A$ is a non-constant map such that $\Phi(0)=0$ and $S\in A$. Two such cycle sets, given by $(A,\Phi,S)$ and $(B,\Phi',S')$, are isomorphic if and only if there exists a group isomorphism $f:A\to B$ such that $\Phi'=f\Phi$ and $f(S)=S'$.
\end{thm}
\begin{proof}
    Note that the elements $\{(a,x)\mid 1\leq a\leq m, x\in A\}$ form a system of representatives for the relation ${\sim}$. We define the map $$\psi: \left(\Z\times 
    \bigoplus_{i\in \Z}\Z \right)/{\sim}\to \Z_m\times \left(\bigoplus_{i\in \Z}\Z \right)/H,$$
    as $\psi([(a,x)])=(a,x)$ for $1\leq a\leq m$. In particular, 
    $$\psi([(0,x)])=\psi([(m,x+s)])=(m,x+s).$$
    Under this identification, we find that the cycle set as given in \cref{thm:jedlicka} is isomorphic to the cycle set on $\Z_m\times (\bigoplus_{i\in \Z}\Z)/H$ given by 
    $$(a,x)\cdot (b,y)=(b-1,y-\chi_0(b)s-c_{a-b}+c_{-b}).$$
    Now instead of starting from $H\leq \bigoplus_{i\in \Z}\Z$ and $s\in \bigoplus_{i\in \Z}\Z/H$ we can also start with an abelian group $A$, $s\in A$ and a group homomorphism $\phi:\bigoplus_{i\in \Z}\Z\to A$; we then set $H=\ker \phi$. As the $c_i$, $i\neq 0$, form a basis of $\bigoplus_{i\in \Z}\Z$, we can freely choose the images $\phi(c_i)\in A$ as long as $\phi(c_i)=\phi(c_{i+m})$. If we denote $\phi(c_i)=\Phi(i)$ we see that every such $\phi$ uniquely corresponds to a map $\Phi:\Z_m\to A$ such that $\Phi(0)=0$. Using $\phi$ to identify $\bigoplus_{i\in \Z}\Z/H$ and $A$ we find a cycle set structure on $\Z_m\times A $ given by
    $$(a,x)\cdot (b,y)=(b-1,y-\chi_0(b)s-\Phi(a-b)+\Phi(-b)).$$
    Recall that different choices of $H$ and $s$ give non-isomorphic cycle sets. It is clear that for two abelian groups $A,B$, $s\in A$, $s'\in B$, and maps $\Phi:\Z_m\to A$ and $\Phi':\Z_m\to B$, the associated homomorphism $\phi$ and $\phi'$ have the same kernel $H$ if and only if there exists a group isomorphism $f:A\to B$ such that $\phi'=f\phi$, or equivalently $\Phi'=f\Phi$. Moreover, $s$ and $s'$ correspond to the same element in $\bigoplus_{i\in \Z}\Z/H$ precisely if $f(s)=s'$. 
    
    At last, define $g:\Z_m\to A$ as 
    $g(b)=\sum_{i=0}^{b-1} \Phi(i)$ for $1\leq b\leq m$.
   If $b\neq 0$, then $g(b+1)-g(b)=\Phi(b)$ and if $b=0$ then $g(b+1)-g(b)=-\sum_{i=1}^{m-1}\Phi(i)$.
   Under the permutation $\theta:(a,x)\mapsto (-a,-x+g(a))$ the cycle set structure now becomes
    \begin{align*}
        \theta^{-1}(\theta(a,x)\cdot \theta(b,y))&=\theta^{-1}((-a,-x+g(a))\cdot (-b,-y+g(b)))\\
        &=\theta^{-1}(-b-1,-y+g(b)-\chi_0(b)s-\Phi(b-a)+\Phi(b))\\
        &=(b+1,y-g(b)+\chi_0(b)s+\Phi(b-a)-\Phi(b)+g(b+1))\\
        &=(b+1,y+\chi_0(b)S+\Phi(b-a)),
    \end{align*}
    where $S=s-\sum_{i=0}^{m-1}\Phi(i)$. To conclude the proof, note that if we are given an abelian group $B$, an isomorphism $f:A\to B$ and we set $\Phi'=f\Phi$, then $f(s)=s'$ if and only if $f(S)=f(s-\sum_{i=0}^{m-1}\Phi(i))=s'-\sum_{i=1}^{m-1}\Phi'(i)=S'$.
\end{proof}
\begin{rem}
    Note that the solutions of the YBE on $X = \Z_p \times A$ corresponding to the cycle sets in \cref{thm:jedlicka_improved} are given by
        \[
        r\begin{pmatrix}
            (a,x) \\ (b,y)
        \end{pmatrix} = \begin{pmatrix}
             ( b-1,y- \chi_0(b-1)S - \Phi(b-1-a))\\
             (a+1, x + \chi_0(a)S + \Phi(a-b+1))
        \end{pmatrix}.
        \]
\end{rem}
\begin{cor} \label{cor:classification_p2_multipermutation}
    Let $X$ be a retractable indecomposable cycle set of size $p^2$, $p$ a prime, then $X$ has finite multipermutation level and is isomorphic to
    one of the following.
    \begin{enumerate}
        \item $X=\Z_{p^2}$ with $x\cdot y=y+1$.
        \item $X=\Z_p\times \Z_p$, with 
        $$(a,x)\cdot (b,y)=(b+1,y+\chi_0(b)S+\Phi(b-a)),$$
        where $\Phi:\Z_p\to \Z_p$ is a non-constant map such that $\Phi(0)=0$ and $S\in \Z_p$. The parameters $S,\Phi$ and $S',\Phi'$ define isomorphic cycle sets if and only if $S'=\alpha S$ and $\Phi'=\alpha \Phi$ for some $\alpha \in \Z_p^*$.
    \end{enumerate}
\end{cor}
\begin{proof}
    It follows from \cref{pro:fibres_have_same_cardinality} that $|\Ret(X)|\in \{1,p\}$. If $|\Ret(X)|=1$ then $X$ clearly has finite multipermutation level and is isomorphic to the given cyclic cycle set on $\Z_{p^2}$. If $|\Ret(X)|=p$, then we know that $\Ret(X)$ is isomorphic to the cycle set on $\Z_p$ with $x\cdot y=y+1$ \cite[Theorem 2.13]{ESS_YangBaxter}. In particular, $|\Ret(\Ret(X))|=1$ and thus $X$ has multipermutation level 2. The statement now follows directly from \cref{thm:jedlicka_improved}.
\end{proof}

\section{Indecomposable irretractable cycle sets of size \texorpdfstring{$p^2$}{p2} - The \texorpdfstring{$p$}{p}-group case} \label{sec:irretractable_cycle_sets_sylow}
The aim of this section is to 
find all isomorphism classes of indecomposable irretractable cycle sets $(X,\cdot)$
of size $p^2$ such that $\Gbr(X)$ is a $p$-group, for $p$ a prime number.

We first construct these cycle sets in \cref{subsec:construction}. In \cref{subsec:redundancy} we then study when such cycle sets are isomorphic and moreover we determine their automorphism groups.

\subsection{Constructing the solutions}\label{subsec:construction}
We first introduce the following result.
\begin{pro} \label{prop:Gmult_not_abelian}
Let $X$ be an irretractable cycle set such that $\G(X)$ is a $p$-brace, then $\Soc(\G(X))=0$ and $\G(X)^\circ$ is not abelian.
\end{pro}
\begin{proof}
    As $X$ is irretractable, it follows from \cite[Lemma 2.1]{BCJO_irretrsqfree} that $\Soc(\G(X))=0$. From \cref{lem:nontrivial_fix} we then find that $\G(X)^\circ$ is not abelian.
\end{proof}

In the remainder of the section, we let $X$ be an indecomposable irretractable cycle set of size $p^2$ and assume that $\G=\G(X)$ is a $p$-group. We associate $X$ with its image in $\G$, which is a transitive cycle base of $\G$.

By \cref{lem:lucchini} and \cref{prop:Gmult_not_abelian} we have a unique system of imprimitivity, for $x\in X$ we denote by $\mathcal{B}_x$ the block containing $x$. We denote by $\mathcal{A}=\G\cap \Z_p^p$ the abelian subgroup of $\Gmult$ which fixes the blocks setwise. Note that $[\G^\circ :\mathcal{A}]=p$.

\begin{pro}
$\Fix(\Gbr) \cap \mathcal{A} = 0$.
\end{pro}

\begin{proof}
Suppose that $0 \neq f \in \Fix(\Gbr) \cap \mathcal{A}$, without loss of generality we may assume that $f^{\circ p}=0$.
By \cref{lem:lambda_of_fix_is_conjugation} and the abelianity of $\mathcal{A}$,
we see that $\lambda_f(g) = {}^fg = g$, for all $g \in \mathcal{A}$. 
Thus $\lambda_f$ fixes all elements in $\mathcal{A}$. 
Moreover, by \cref{lem:fixed_points_are_not_central}, $ \Fix(\Gbr) \cap Z(\Gmult)\subseteq \Soc(\Gmult)=0$, so $f \not\in Z(\Gmult)$.
Hence, knowing also that $|\Gbr|/|\mathcal{A}| = p$, we see that $\mathcal{A} = \{ g \in \Gbr\mid {}^fg = g \}$.
As $\lambda_f$ is an automorphism, $\mathcal{A}$ is a subbrace of index $p$ in $\Gbr$. By \cref{pro:maximal_subbraces_of_p_braces_are_ideals}, $\mathcal{A}$ is an ideal of $\Gbr$.

As $\Gbr/\mathcal{A}$ is a brace of order $p$, it must be trivial and the canonical map $\Gbr\twoheadrightarrow \Gbr/\mathcal{A}$ maps the transitive cycle base $X$ to a single element. Hence there exists some $g\in \Gbr$ such that for all $x\in X$ there exists some $a_x\in \mathcal{A}$ such that $x=g+a_x$. In particular, if we set $\gamma=f*g\in \mathcal{A}$, we find that also $f*x=f*g+f*a_x=\gamma$.
Hence inductively for all $n\in \N$ we find $f^{\circ n}*x=n\gamma$ since
\begin{align*}
f^{\circ n}*x&
=\lambda_{f^{\circ n}}(x)-x
=\lambda_f(\lambda_{f^{\circ n-1}}(x))-\lambda_f(x)+\lambda_f(x)-x\\
&=\lambda_f(f^{\circ n-1}*x)+f*x
=\lambda_f((n-1)\gamma)+\gamma=n\gamma.  
\end{align*}
In particular, this implies that $p\gamma=0$.
If $\gamma=0$, then $\lambda_f(x)=x$, for all $x\in X$ and therefore $\lambda_f=\id_\Gbr$, which would imply that $f\in \Soc(\Gbr)$. It therefore follows that $\gamma\neq 0$, which in turn implies that $\lambda_f$ has no fixed points on $X$. In particular, for each $a\in \mathcal{A}$ and $x\in X$, there exists some $n\in \N$ such that $\lambda_a(x)=\lambda_f^n(x)$, hence $a*x=f^{\circ n}*x=n\gamma$.

Now consider $I$ be the subgroup of $\Gadd$ generated by $\{ a*g\mid a\in \mathcal{A},g\in \G\}$, which is an ideal by
\cite[Corollary after Proposition 6]{Rump_braces}. Because $X$ generates $\Gadd$ and $*$ is left distributive, $I$ is the subgroup of $\Gadd$ generated by $\{a*x\mid a\in \mathcal{A},x\in X\}$. By the previous discussion, $I=\{0,\gamma,\ldots,(p-1)\gamma\}$. As $|I|=p$, we find that $I\subseteq \Fix(\Gbr)$. Since $I$ is a minimal normal subgroup of the nilpotent group $\Gbr^\circ$, we also find that $I\subseteq Z(\Gbr^\circ)$, hence \cref{lem:fixed_points_are_not_central} yields a contradiction with the assumption that $\Soc(\Gbr)=0$.
\end{proof}

Now that we know that $\Fix(\Gbr) \cap \mathcal{A} = 0$, it follows from 
\cref{lem:nontrivial_fix} that $|\Fix(\Gbr)| = p$. Also, we have a semidirect product $\Gmult = \mathcal{A} \rtimes \Fix(\Gbr)^\circ$. We use this for the following construction: for all $g \in \Gbr$, we have an equality $g\circ \Fix(\Gbr) = g + \lambda_{g}(\Fix(\Gbr)) = g+ \Fix(\Gbr)$. Therefore, $\mathcal{A}$ forms a system of representatives for $\Gmult / \Fix(\Gbr)^{\circ}$ and $\Gadd / \Fix(\Gbr)^+$. We keep the multiplication on $\mathcal{A}$ as it is, but as $\mathcal{A}$ is not necessarily closed under $+$, we define $g \oplus h$ as the unique element in $\mathcal{A} \cap (g + h + \Fix(\Gbr))$.

\begin{pro}
The structure $(\mathcal{A},\oplus,\circ)$ is a brace.
\end{pro}

\begin{proof}
For $a,b,c \in \mathcal{A}$, we calculate
\begin{align*}
    \{a\circ b\ominus a\oplus a\circ c\}&
=\mathcal{A} \cap (a\circ b-a+a\circ c + \Fix(\Gbr))
=\mathcal{A} \cap (a\circ (b + c) + \Fix(\Gbr))\\
&=\mathcal{A} \cap \big((a+ \Fix(\Gbr)) \circ (b+c + \Fix(\Gbr))\big)\\
&=\big(\mathcal{A} \cap (a+ \Fix(\Gbr)\big) \circ \big(\mathcal{A} \cap (b+c + \Fix(\Gbr)\big)\\
&=\big(\mathcal{A} \cap (a+ \Fix(\Gbr)\big) \circ \{b\oplus c\}
=\{a\circ(b\oplus c)\}.
\end{align*}
Hence $a\circ(b\oplus c)=a\circ b\ominus a\oplus a\circ c$.
\end{proof}
Denote by $\Tilde{\mathcal{A}}$ the thus constructed brace on $\mathcal{A}$. We define the \emph{transition} map
\begin{align*}
    \tau: \Gbr & \twoheadrightarrow \Tilde{\mathcal{A}} \\
    g & \mapsto [g], \textnormal{ where } \{ [g] \} = \mathcal{A} \cap (g\circ \Fix(\Gbr)).
\end{align*}
Note that $\tau$ is a group homomorphism $\Gbr^+ \to \Tilde{\mathcal{A}}^\oplus$ but this is not necessarily true for $\Gbr^\circ \to \Tilde{\mathcal{A}}^\circ$. With this notation, we also see that $g \oplus h = [g + h]$. Furthermore, on $\Tilde{\mathcal{A}}$, the $\lambda$-action changes to $\Tilde{\lambda}_g(h) = [\lambda_g(h)]$ which implies that the image $[X] \subseteq \Tilde{\mathcal{A}}$ is invariant under its $\Tilde{\lambda}$-action (which is not necessarily transitive). Therefore, it is again a cycle set under the operation $[x] \Tilde{\cdot} [y] = \Tilde{\lambda}^{-1}_{[x]}([y])$. Denote this cycle set by $\Tilde{X}$.

\begin{pro} \label{pro:fibres_have_size_p}
    $|\Tilde{X}| = p$.
\end{pro}

\begin{proof}
Note that $\Gmult$ still acts transitively on $\Tilde{X}$ by $\lambda_g([x]) = [\lambda_g(x)]$ and that $\tau^{-1}([x]) = (x + \Fix(\Gbr)) \cap X$. In particular, $|\tau^{-1}([x])| \leq p$ and thus $|\Tilde{X}| \in \{p,p^2 \}$.

Suppose that $\tau:X\to \Tilde{X}$ is injective, then $\Gbr(\Tilde{X})^{\circ}$ is isomorphic to $\mathcal{A}$ and each element of $\Tilde{X}$ acts differently on $\Tilde{X}$ by the $\Tilde{\lambda}$-action, so $\Tilde{X}$ is irretractable. Hence, this contradicts \cref{prop:Gmult_not_abelian}. Therefore, $|\Tilde{X}| = p$.
\end{proof}

\begin{pro} \label{pro:blocks_are_cosets_of_fix}
    For all $x \in X$, we have the equality $\mathcal{B}_x = x + \Fix(\Gbr) = x\circ\Fix(\Gbr)$.
    In particular, each block intersects $\mathcal{A}$ in precisely one element.
\end{pro}

\begin{proof}
    By \cref{pro:fibres_have_size_p}, for all $x \in X$, we have $x + \Fix(\Gbr) \subseteq X$. As $\Gmult$ leaves $\Fix(\Gbr)$ invariant under the $\lambda$-action, we deduce that the cosets $x + \Fix(\Gbr)$ form a non-trivial system of imprimitivity for the $\lambda$-action of $\Gmult$ on $X$. The uniqueness of such a system (\cref{lem:lucchini}) implies that $\mathcal{B}_x = x + \Fix(\Gbr)$.

    Moreover, it was observed earlier that $\mathcal{A}\cap (g+\Fix(\G))$ is a singleton for each $g\in \G$. Together with the first part of this proposition, this gives the last part of the statement. 
\end{proof}

\begin{pro} \label{pro:X_is_a_double_coset_in_GX}
Let $x \in X$, then $X = \Fix(\Gbr)\circ x\circ \Fix(\Gbr)$.
\end{pro}

\begin{proof}
Let $0 \neq f \in \Fix(\Gbr)$. Then ${}^f\mathcal{B}_x = \lambda_f(\mathcal{B}_x) \neq \mathcal{B}_x$ as $f \not\in \mathcal{A}$. Therefore, $\Fix(\Gbr)$ acts transitively on the system of blocks by conjugation. Furthermore, by \cref{pro:blocks_are_cosets_of_fix}, $\Fix(\Gbr)$ acts transitively on every single block by right-multiplication. Therefore, each $y \in X$ is of the form $y = {}^fx\circ f^{\prime}$ for some $f,f^{\prime} \in \Fix(\Gbr)$, which can be rewritten as $y = f\circ x\circ f^{\prime}\circ \overline{f}$.
\end{proof}

From \cref{pro:X_is_a_double_coset_in_GX} it therefore follows that we can coordinatize the elements of $X$ in the following way: by \cref{pro:blocks_are_cosets_of_fix} we can choose an element $\alpha_{0,0} \in X \cap \mathcal{A}$ and some $0 \neq f \in \Fix(\Gbr)$. For $(a,x) \in \Z_p \times \Z_p$, we then set $\alpha_{a,x} = \left({}^{\overline{f}^{\circ a}}\alpha_{0,0}\right)\circ f^{\circ x}$. By \cref{pro:X_is_a_double_coset_in_GX}, this gives a unique coordinatization of the elements in $X$.

$\lambda_{\alpha_{0,0}}$ leaves all blocks invariant, therefore we can write $\lambda_{\alpha_{0,0}}({\alpha_{b,0}}) = \alpha_{b,-\Phi(b)}$ for some map $\Phi: \Z_p \to \Z_p$. We can now calculate:
\begin{align*}
\alpha_{0,0} \cdot \alpha_{b,y} & =  \lambda_{\alpha_{0,0}}^{-1}(\alpha_{b,-\Phi(b)} + (y+\Phi(b))f) \\
& = \alpha_{b,0} + (y+\Phi(b))f \\
& = \alpha_{b,y + \Phi(b)}.
\end{align*}
Also, we need to determine
\[
f^{\circ a} \cdot \alpha_{b,y} = \lambda_{\overline{f}^{\circ a}}\left({}^{\overline{f}^{\circ b}}\alpha_{0,0}\circ f^{\circ y}\right) = {}^{\overline{f}^{\circ (a+b)}}\alpha_{0,0}\circ f^{\circ y} = \alpha_{a+b,y}. 
\]

We write $\alpha_{a,x} = \overline{f}^{\circ a}\circ\alpha_{0,0}\circ f^{\circ (a+x)}$ and calculate:
\begin{align*}
\alpha_{a,x} \cdot \alpha_{b,y} & = f^{\circ (a+x)} \cdot ( \alpha_{0,0} \cdot ( \overline{f}^{\circ a} \cdot \alpha_{b,y})) \\
& = f^{\circ (a+x)} \cdot ( \alpha_{0,0} \cdot \alpha_{b-a,y} ) \\
& = f^{\circ (a+x)} \cdot \alpha_{b-a, y + \Phi(b-a)} \\
& = \alpha_{b+x,y + \Phi(b-a)}.
\end{align*}

\begin{thm} \label{thm:irretractable_cycle_sets_with_p_brace}
Let $X$ be an indecomposable, irretractable cycle set of size $p^2$ with $\Gbr(X)$ a $p$-group, then $X$ is isomorphic to a cycle set of the form $X = \Z_p \times \Z_p$ with the multiplication
\[
(a,x) \cdot (b,y) = (b+x, y + \Phi(b-a))
\]
where $\Phi: \Z_p \to \Z_p$ is a non-constant map with $\Phi(A) = \Phi(-A)$, for all $A \in \Z_p$. Vice versa, this construction always results in an indecomposable, irretractable cycle set with $\Gbr(X)$ a $p$-group.
\end{thm}

\begin{proof}
In the preceding calculations, we have already established that the given multiplication rule is necessary. We now determine
\begin{align*}
((a,x) \cdot (b,y)) \cdot ((a,x) \cdot (c,z)) & = (b+x, y + \Phi(b-a)) \cdot (c+x, z + \Phi(c-a)) \\
& = (c+x + y + \Phi(b-a), z + \Phi(c-a) + \Phi(c-b)).
\end{align*}
Similarly,
\[
((b,y) \cdot (a,x)) \cdot ((b,y) \cdot (c,z)) = (c+y + x + \Phi(a-b), z + \Phi(c-b) + \Phi(c-a)).
\]
A comparison shows that in order for $X$ to satisfy \cref{eq:c1_cycloid_equation}, $\Phi(b-a) = \Phi(a-b)$ must hold for all $a,b \in \Z_p$ which amounts to saying that $\Phi(A) = \Phi(-A)$, for all $A \in \Z_p$. By the same calculation, one sees that this obstruction to $\Phi$ is sufficient for $X$ to satisfy \cref{eq:c1_cycloid_equation}. By construction, all maps $\sigma_{(a,x)}$ are bijective. Furthermore, the square map
\[
\Sq(a,x) = (a,x) \cdot (a,x) = (a+x, x + \Phi(0))
\]
is also quickly seen to be bijective. Finally, irretractability is the same as saying that for any $a,a^{\prime} \in \Z_p$, there is at least one $b \in \Z_p$ such that $\Phi(b-a) = \Phi(b-a^{\prime})$. But this is clearly equivalent to $\Phi$ not being constant.

Finally, note that $\G(X)$ clearly acts transitively on the system of blocks $\{a\} \times \Z_p$. For $b \in \Z_p$ with $\Phi(b) \neq 0$, we see that $(0,0) \cdot (b,y) = (b, y + \Phi(b)) \neq (b,y)$ which shows that the $\G(X)$-orbit of $(b,y)$ contains at least the block $\{ b \} \times \Z_p$. This implies that $\G(X)$ acts transitively on $X$, hence $X$ is indecomposable. Also note that $\G(X)$ is contained in $\Z_p\wr \Z_p\leq \Sym_{\Z_p\times \Z_p}$, hence $\G(X)$ is a $p$-group.
\end{proof}
We note the following corollary, which will be useful later in \cref{sec: controlling p'}:
\begin{cor}\label{cor:blocks_intersect_A_generate}
    The elements in $X\cap \mathcal{A} $ generate the whole cycle set.
\end{cor}
\begin{proof}
Note that, using the explicit form in \cref{thm:irretractable_cycle_sets_with_p_brace}, we are considering
the set $X\cap \mathcal{A}=\{(a,0)\mid a\in \Z_p\}$.
    As $\Phi$ is non-constant, we easily see that these elements generate the whole cycle set $X$.
\end{proof}

\subsection{Getting rid of redundancy and determining automorphisms} \label{subsec:redundancy}

The aim of this subsection is to determine unique representatives for the irretractable cycle sets determined in \cref{subsec:construction} and moreover, describe their automorphism groups.

Let $\mathcal{F}_p$ be the set of all non-constant functions $\Phi: \Z_p \to \Z_p$ with the property that $\Phi(A) = \Phi(-A)$, for all $A \in \Z_p$. $\mathcal{F}_p$ is acted upon by $\Z_p^{\ast}$ via $({}^{\alpha} \Phi)(A) = \alpha^{-1} \Phi(\alpha A)$. From now on, let $\mathcal{R}_p$ be a fixed system of representatives for this action.

Recall from \cref{thm:irretractable_cycle_sets_with_p_brace} that the cycle sets in the considered case are described as $\Z_p \times \Z_p$ with the operation
\[
(a,x) \cdot (b,y) = (b+x, y + \Phi(b-a)),
\]
where $\Phi \in \mathcal{F}_p$.

By \cref{prop:Gmult_not_abelian}, $\Gbr(X)^\circ$ is non-abelian and, by \cref{lem:lucchini}, has a unique non-trivial system of imprimitivity that consists of the blocks $\{a \} \times \Z_p$. Assume that $X = \Z_p \times \Z_p$ comes with two cycle set operations $\cdot, \Pri{\cdot}$ that are given by the parameters $\Phi, \Pri{\Phi} \in \mathcal{F}_p$ and that $\phi: (X,\cdot) \to (X,\Pri{\cdot})$ is an isomorphism. Then $\phi\G(X,\cdot)\phi^{-1}=\G(X,\cdot')$ so in particular $\phi$ must normalize the cyclic permutation action on the blocks and thus be of the form $\phi(a,x) = (\alpha a + \beta, \pi_a(x))$ for some $\pi_a \in \Sym_{\Z_p}$, $\alpha \in \Z_p^*$, $\beta \in \Z_p$. We now calculate
\begin{align*}
    \phi((a,x) \cdot (b,y)) & = (\alpha (b + x) + \beta , \pi_{b+x}(y + \Phi(b-a))), \\
    \phi(a,x) \Pri{\cdot} \phi(b,y) & = ( \alpha b + \beta + \pi_a(x), \pi_b(y) + \Pri{\Phi}(\alpha(b-a)).
\end{align*}
Equating these terms results in $\pi_a(x) = \alpha x$, considering the first coordinate. Taking this into account when considering the second coordinate leaves us with the equation
\[
\alpha (y + \Phi(b-a)) = \alpha y + \Pri{\Phi}(\alpha(b-a)) \Leftrightarrow \Phi(b-a) = \alpha^{-1}\Phi'(\alpha(b-a)).
\]
This shows that $\Phi, \Pri{\Phi}$ define isomorphic cycle sets if and only if there is an $\alpha \in \Z_p^*$ such that $\Phi={}^{\alpha}\Phi'$. Putting $\Phi = \Pri{\Phi}$, the same considerations prove that $\phi$ provides an automorphism of a solution with parameter $\Phi$ if and only if $\phi(a,x) = (\alpha x + \beta, \alpha x)$ for $\alpha \in \Z_p^{\ast}$, $\beta \in \Z_p$ with ${}^{\alpha}\Phi = \Phi$.

We conclude:

\begin{thm} \label{thm:isoclasses_of_irretractable_cyclesets_sylow_case}
Let $X$ be an indecomposable irretractable cycle set of size $p^2$ where $p$ is a prime.
\begin{enumerate}
    \item If $\mathcal{G}(X)$ is a $p$-group, then there is a unique $\Phi \in \mathcal{R}_p$ such that $X$ is isomorphic to the cycle set on $X = \Z_p \times \Z_p$ with multiplication
    \[
    (a,x) \cdot (b,y) = (b+x, y + \Phi(b-a)).
    \]
    \item Let $X, \Phi$ be as in the previous item. Then any automorphism of $X$ is of the form $(a,x) \mapsto (\alpha a + \beta, \alpha x)$ for some $\alpha \in \Z_p^{\ast}$, $\beta \in \Z_p$ with ${}^{\alpha}\Phi = \Phi$.
\end{enumerate}
\end{thm}

\section{Indecomposable irretractable cycle sets of size \texorpdfstring{$p^2$}{p2} - The general case}\label{sec: controlling p'}

In this section, we will focus on the general problem.
More precisely, we will construct all indecomposable irretractable cycle sets $X$ of size $p^2$, where $p$ is a prime number.

Before restricting to this specific case, we prove a useful lemma. Given a brace $A$ and a subset $S\subseteq A$ we define
$$\Fix_A(S)=\{a\in A\mid \lambda_s(a)=a \text{ for all }s\in S\}.$$
In fact, it follows from \cref{lem:lambda_of_fix_is_conjugation} that $\lambda_{\overline{a}}(^{\overline{a}}s)=-a+s+\lambda_s(a)$, so $\lambda_s(a)=a$ if and only if $\lambda_a(^{\overline{a}}s)=s$, which is equivalent to $\lambda_{\overline{a}}(s)={}^{\overline{a}}s$. So alternatively $$\Fix_A(S)=\{a\in A\mid \lambda_{\overline{a}}(s)={}^{\overline{a}}s\text{ for all }s\in S\}.$$

\begin{lem}\label{lem:Fix_is_brace}
    Let $A$ be a finite brace, $L$ a left ideal of $A$ and $G$ a normal subgroup of $A^\circ$, then $\Fix_A(L\cap G)$
    is a subbrace of $A$ contained in the normaliser of $L\cap G$ in $A^\circ$
\end{lem}
\begin{proof}
    From the original definition of $\Fix_A(L\cap G)$ we see that $\Fix_A(L\cap G)^+$ is a group. 
    Let $a\in \Fix_A(L\cap G)$ and $b\in L\cap G$. Then $\lambda_{\overline{a}}(b)\in L$, but as $\lambda_{\overline{a}}(b)={}^{\overline{a}}b$, also $\lambda_{\overline{a}}(b)\in G$. We conclude that $\lambda_{\overline{a}}(b)\in L\cap G$. Using the alternative description of $\Fix_A(L\cap G)$ we now see that $\Fix_A(L\cap G)$ is closed under the $\circ$-operation and non-empty. As $A$ is finite, we conclude that $\Fix_A(L\cap G)^\circ$ is a subgroup of $A^{\circ}$. Therefore $\Fix_A(L\cap G)$ is a subbrace. As ${}^{\overline{a}}b\in L\cap G$ for all $a \in \Fix_A(L\cap G)$, $b \in L \cap G$ we find indeed that $\Fix_A(L\cap G)$ is contained in the normaliser of $L\cap G$ in $A^\circ$.
\end{proof}

In the remainder of the section, we let $X$ be an indecomposable irretractable cycle set of size $p^2$ and $\G=\G(X)$. As we already covered the case where $\G$ is a $p$-group in \cref{sec:irretractable_cycle_sets_sylow}, we can assume that $\G$ is not a $p$-group, but recall that $\G$ is solvable by \cite[Theorem 2.15]{ESS_YangBaxter}. We associate $X$ with its image in $\G$, which is a transitive cycle base. As $X$ is irretractable, it follows from \cite[Lemma 2.1]{BCJO_irretrsqfree} that $\Soc(\G)=0$.

Let $\G_p$ be the Sylow $p$-subgroup of $\G^+$ and let $\G_{p'}$ be the Hall $p'$-subgroup of $\G^+$, both are characteristic in $\G^+$ hence they are left ideals of $\G$. We denote $X=\{x_1,\ldots,x_{p^2}\}$ and for $1\leq i\leq p^2$ we define $y_i\in \G_p$ and $z_i\in \G_{p'}$ such that $x_i=y_i+z_i$. As the $\lambda$-action of $\G_p^\circ$ is transitive on $X$, and therefore also on $Y=\{y_i\mid 1\leq i\leq p^2\}$, we find that $Y$ is a transitive cycle base of the brace $\G_p$. In particular, this implies that $|Y|\in \{1,p,p^2\}$.

Let $q \neq p$ be prime and $\G_q$ the $q$-Sylow subgroup of $\Gadd$ which is a left ideal by the same argument as above. $\Fix_\G(\G_q)$ is a brace by \cref{lem:Fix_is_brace} and $Y\cap \Fix_\G(\G_q)$ is a sub-cycle set of $Y$. Now if $Y$ has finite multipermutation level, then $Y\cap \Fix_\G(\G_q)$ is either empty or equal to $Y$ by \cite[Theorem 5.1]{castelli2023studying}. As $|Y|$ is a $p$-power, $\G_q$ fixes at least one point in $Y$ under the $\lambda$-action. This means that $Y\cap \Fix_\G(\G_q)=Y$ and thus $\G_p\subseteq \Fix_\G(\G_q)$. It follows that $\G_p^\circ$ normalizes $\G_q^\circ$ but as $\G_p$ acts faithful and transitive on a set of size $p^2$, this implies that $\G_q=0$ and hence, $\G_{p^{\prime}} = 0$ but this contradicts the assumption that $\G$ is not a $p$-group. We therefore deduce that $Y$ is not of finite multipermutation level. Together with the earlier observation that $|Y|\in \{1,p,p^2\}$, we conclude that $Y$ is irretractable of size $p^2$ and therefore as described in \cref{thm:irretractable_cycle_sets_with_p_brace}. In particular, we find that $\G_p^\circ$ is not abelian.

From now on we consider the unique block system of $X$ under the action of $\G$, and recall that the uniqueness is guaranteed by \cref{lem:lucchini}. As before we denote this by $\{\mathcal{B}_x\mid x\in X\}$. We denote the subgroup of $\G^\circ$ that fixes the blocks setwise by $\mathcal{A}$. Also, we define $\mathcal{A}_p=\mathcal{A}\cap \G_p$ and $\mathcal{A}_{p'}=\mathcal{A}\cap \G_{p'}$. Note that $\mathcal{A}_p$ is a $p$-Sylow subgroup of $\mathcal{A}$ and $\mathcal{A}_{p'}$ is a Hall $p'$-subgroup of $\mathcal{A}$. In particular, $\mathcal{A}_p$ is normal in $\G^\circ$ by \cref{pro:impritive_solvable_action_on_Z2}.

Let $a\in \mathcal{A}_p$ and $g\in \G_{p'}$, then \cref{lem:lambda_of_fix_is_conjugation} yields $\lambda_g({}^{\overline{g}}a)=-g+a+\lambda_a(g)$ hence $-a+\lambda_g({}^{\overline{g}}a)=-g+\lambda_a(g)$. As $-a+\lambda_g({}^{\overline{g}}a)$ is contained in $\G_p$ and $-g+\lambda_a(g)$ is contained in $\G_{p'}$, we find that $\lambda_a(g)=g$. By \cref{lem:lambda_of_fix_is_conjugation} this implies that ${}^ga=\lambda_g(a)$, so the $\lambda$-action of $\G_{p'}$ restricts to $\mathcal{A}_p$, so also to $Y\cap \mathcal{A}_p$.

By \cref{pro:blocks_are_cosets_of_fix} we know that $\mathcal{A}_p$ contains a unique representative of each block in the block system of $Y$ under the action of $\G_{p^{\prime}}^\circ$, hence also under the action of $\G^\circ$. This means that $\mathcal{A}_{p'}$ acts trivially on the set $Y\cap \mathcal{A}_p$ and thus $Y\cap \mathcal{A}_p\subseteq \Fix_\G(\mathcal{A}_{p'})$. From \cref{lem:Fix_is_brace} we know that $\Fix_\G(\mathcal{A}_{p'})$ is a subbrace, since $\mathcal{A}_{p'}=\mathcal{A}\cap \G_{p'}$. In particular, $\Fix_\G(\mathcal{A}_{p'})\cap Y$ is a sub-cycle set of $Y$ which contains $Y\cap \mathcal{A}_p$, but from \cref{cor:blocks_intersect_A_generate} it then follows that $\Fix_\G(\mathcal{A}_{p'})\cap Y=Y$. However, as the $\lambda$-action of $\G$ on $Y$ is faithful, this implies that $\mathcal{A}_{p'}=0$.

By \cref{pro:impritive_solvable_action_on_Z2} we find that $|\G_{p'}|<p$ and thus the $\lambda$-action of $\G_p$ on $\G_{p'}$ is trivial. As this same action is transitive on $Z$ we find $|Z|=1$. Let $z\in Z$, then $\lambda_z$ is a brace automorphism of $\G_p$ by \cref{lem:lambda_of_fix_is_conjugation} and therefore also its restriction to $Y$ yields a cycle set automorphism. For any $x_i,x_j\in X$ we find
    $$x_i\cdot x_j=(y_i+z)\cdot (y_j+z)=\lambda_{y_i+z}^{-1}(y_j+z)=\lambda_z^{-1}(y_i\cdot y_j)+z.$$
We find that the cycle set structure on $X$ is obtained by deforming the cycle set structure on $Y$ by an automorphism of $Y$. \cref{lem:deformation_of_cycle_set_by_automorphism} shows that such a deformation is always possible. 

\begin{rem}\label{rem:cabling}
    We remark that the idea of starting from a finite cycle set $X$ and then considering its projection $Y$ onto the $p$-Sylow subgroup $\G_p$, is strongly related to the notion of \emph{cabling}, see \cite{MR4717099,MR4716443}. Recall that for $k\in \Z$, the \emph{$k$-cabled cycle set} is defined as $(X,\cdot_k)$ with 
    $x\cdot_k y=(k\sigma_{x})(y)$, where $k\sigma_x\in \G(X)$ is the $k$-th power of $\sigma_x$ in $(\G(X),+)$. 

    In particular, $\G(X,\cdot_k)$ is a subgroup of $\G(X,\cdot)^\circ$. We now claim that it is even a subbrace. If we denote the $\sigma$-maps of $(X,\cdot_k)$ by $\sigma_x'$, then we find that the addition $+_k$ in $\G(X,\cdot_k)$ is given by 
    \begin{align*}
        \sigma_x'+_k\sigma_y'&=\sigma_x'\circ \sigma_{x\cdot_k y}'\\
        &=(k\sigma_x)\circ (k\sigma_{(k\sigma_x)(y)})\\
        &=(k\sigma_x)+\lambda_{k\sigma_x}(k\sigma_{(k\sigma_x)(y)})\\
        &=(k\sigma_x)+k\lambda_{k\sigma_x}(\sigma_{(k\sigma_x)(y)})\\
        &=(k\sigma_x)+(k\sigma_y)=\sigma_x'+\sigma_y',
    \end{align*}
    where $+$ denotes the addition in $\G(X,\cdot)$.
    
    Let $|\G(X,\cdot)|=p^rm$ where $(p,m)=1$. As the additive $p$-Sylow subgroup is $\G(X,\cdot)_p=\{kg\mid g\in \G\}$ for any non-zero multiple $k$ of $m$, we find that $\G(X,\cdot)_p= \G(X,\cdot_k)$. In particular, $\Ret(X,\cdot_k)$ isomorphic to the cycle set $\{k\sigma_x\mid x\in X\}\subseteq \G(X,\cdot)$ where the equivalence class of $x$ is mapped to $k\sigma_x$. If we let $k$ be such that $k\equiv 1 \pmod{p^r}$ then we find that the $k\sigma_x$ is precisely the projection of $\sigma_x$ onto $\G(X,\cdot)_p$, hence in this case we find that $Y\cong \Ret(X,\cdot_k)$, with $Y$ as before.
\end{rem}
\begin{pro}
    Let $(X,\cdot)$ be an indecomposable cycle set of order $p^n$, with $p$ a prime. Let $k$ be the largest divisor of $|\G(X,\cdot)|$ coprime to $p$. If $(X,\cdot_k)$ has finite multipermutation level, then $\G(X,\cdot)$ is a $p$-group and thus $k=1$.
\end{pro}
\begin{proof}

    The proof is essentially the same as how we proved earlier in this section that $Y$ is not of finite multipermutation level. 

    First of all, as $\G(X,\cdot_k)^\circ$ is a $p$-Sylow subgroup, $(X,\cdot_k)$ is still an indecomposable cycle set. Now let $Y=\{k\sigma_x\mid x\in X\}\cong \Ret(X,\cdot_k)$. As $(X,\cdot_k)$ has finite multipermutation level, so does $Y$. 
    
    Let $q\neq p$ be a prime and let $\G(X,\cdot)_q$ be the $q$-Sylow subgroup of $\G(X,\cdot)^+$. If we consider the $\lambda$-action of $\G(X,\cdot)_q$ on $Y$ we find that it has fixed points. By \cref{lem:Fix_is_brace} we find that $\Fix(\G(X,\cdot)_q)$ is a brace thus $\Fix(\G(X,\cdot)_q)\cap Y$ is a sub-cycle set. As this sub-cycle set is non-empty, \cite[Theorem 5.1]{castelli2023studying} implies that $Y\subseteq \Fix(\G(X,\cdot)_q)$ and hence $\G(X,\cdot)_p\subseteq \Fix(\G(X,\cdot)_q$. As a result, $\G(X,\cdot)_q$ is a normal subgroup of $\G(X,\cdot)^\circ$, but this is impossible as $\G(X,\cdot)^\circ$ acts transitively and faithfully on a set of $p$-power order. We conclude that $\G(X,\cdot)$ is a $p$-group and therefore $k=1$.
\end{proof}

\begin{lem}\label{lem:deformation_of_cycle_set_by_automorphism}
    Let $(X,\cdot)$ be a cycle set and $\phi$ be an automorphism of $(X,\cdot)$. Then the following statements hold:
    \begin{enumerate}
        \item $X$ is a cycle set for the operation
    $$x\cdot_\phi y=\phi(x\cdot y).$$
    \item If $(X,\cdot)$ is irretractable, then so is $(X,\cdot_\phi)$.
    \end{enumerate}
\end{lem}
\begin{proof}
    One can verify directly that $(X,\cdot_\phi)$ satisfies \eqref{eq:c1_cycloid_equation}-\eqref{eq:c3}. However, as this is useful in the proof of \cref{lem:deformation_improved} we construct $(X,\cdot_\phi)$ as the sub-cycle set of a cycle set coming from a brace. 
    
    By the functoriality of $G(X,r)$ we get an induced automorphism $\phi'$ of $G(X,r)$ which restricts to $\phi$ on the generating set $X\subseteq G(X,\cdot)$. We let $\Z$ denote the trivial brace on $\Z$ and we let $\Z^\circ$ act on $G(X,\cdot)$ by $1\cdot g=(\phi')^{-1}(g)$ for $g\in G(X,\cdot)$. Hence we can construct the semi-direct product $G(X,\cdot)\rtimes \Z$ in the sense of \cite[Corollary 3.37]{SV_OnSkewBraces}. The set $X+1\subseteq G(X,\cdot)\rtimes \Z$ is closed under the $\lambda$-action. Hence it is a sub-cycle set which is precisely the cycle set in the statement under the correspondence $x\mapsto x+1$.

    If $(X,\cdot)$ is irretractable, then also $(X,\cdot_\phi)$ is irretractable as it follows directly that $x\cdot z=y\cdot z$ if and only if $x\cdot_\phi z=y\cdot_\phi z$ for $x,y,z\in X$.
\end{proof}
\begin{lem}\label{lem:deformation_improved}
    Let $(X,\cdot)$ be a finite cycle set and $\phi$ an automorphism of $(X,\cdot)$ of order $m$ coprime to $|\G(X)|$ such that $\phi$ has a fixed point. Then $\G(X,\cdot_\phi)^\circ = \G(X,\cdot)^\circ \rtimes \langle \phi\rangle$ as subgroups of $\Sym_X$ and $\G(X,\cdot_\phi)\cong \G(X,\cdot)\rtimes \Z_m$ as braces where $\Z_m$ is given the trivial brace structure and acts on $\G(X,\cdot)$ by $\phi$. 
    
    In particular, if $(X,\cdot)$ is indecomposable then so is $(X,\cdot_\phi)$ and $(X,\cdot)=(X,(\cdot_\phi)_k)$ for any $k\in \Z$ such that $k\equiv 0\pmod{m}$ and $k\equiv 1\pmod {|\G(X)|}$.
\end{lem}
\begin{proof}
    Let $G(X,\cdot)\rtimes \Z$ be as described in the proof of \cref{lem:deformation_of_cycle_set_by_automorphism}. The universal property of $G(X,\cdot_\phi)$ provides a brace homomorphism $f:G(X,\cdot_\phi)\to G(X,\cdot)\rtimes \Z$ mapping $x\to x+1$. The canonical map $g:(G(X,\cdot)\rtimes \Z)^\circ\to \Sym_X$ sending $x\mapsto \sigma^{-1}_x$ and $1\mapsto \phi^{-1}$ is a group homomorphism.  Clearly $\ker gf=\Soc(G(X,\cdot_\phi))$ and the image is precisely $\G(X,\cdot_\phi)$. As $\G(X,\cdot)\cap \langle \phi\rangle=\{\id\}$ we find that $\G(X,\cdot_\phi)^\circ=\G(X,\cdot)^\circ \rtimes \langle \phi\rangle$. Also, as $\ker g= \Soc(G(X,\cdot))\times m\Z$, we find that $\G(X,\cdot_\phi)\cong \G(X,\cdot)\rtimes \Z_m$ as braces.

    It follows directly that if $\G(X,\cdot)$ acts transitively on $X$ then so does $\G(X,\cdot_\phi)$.

    As $\G(X,\cdot)$ is a Hall-subgroup of $\G(X,\cdot_\phi)^+$, the last part of the statement follows from \cref{rem:cabling}. 
\end{proof}
For an element $g$ of a finite group $G$, we denote its order by $o(g)$.
\begin{lem}\label{lem:isomorphism_deformed}
    Let $(X,\cdot)$, $(X,\cdot')$ be finite cycle sets and $\phi\in \Aut(X,\cdot)$, $\psi\in \Aut(X,\cdot')$ such that $\gcd(|\G(X,\cdot)||\G(X,\cdot')|,o(\phi)o(\psi))=1$. Then $f:(X,\cdot_\phi)\to (X,\cdot'_\psi)$ is an isomorphism if and only if $f:(X,\cdot)\to (X,\cdot')$ is an isomorphism and $\psi=f\phi f^{-1}$. In particular, $\Aut(X,\cdot_\phi)$ is precisely the centraliser of $\phi$ in $\Aut(X,\cdot)$.
\end{lem}
\begin{proof}
    Assume that $f:(X,\cdot_\phi)\to (X,\cdot'_\psi)$ is an isomorphism. From \cref{lem:deformation_improved} and the assumptions we find the existence of some $k\in \Z$ such that $(X,(\cdot_\phi)_k)=(X,\cdot)$ and $(X,(\cdot'_\psi)_k)=(X,\cdot')$. The functoriality of cabling now yields that $f$ induces an isomorphism $f:(X,\cdot)\to (X,\cdot')$. For any $x,y\in X$ we find $f(x\cdot_\phi y)=f\phi(x\cdot y)$ and $f(x)\cdot_\psi f(y)=\psi f(x \cdot y)$, hence $f(x\cdot_\phi y)=f(x)\cdot_\psi f(y)$ if and only if $\psi=f\phi f^{-1}$. This proves one implication of the statement. 

    Assume that $f:(X,\cdot)\to (X,\cdot')$ is an isomorphism and $\psi=f\phi f^{-1}$. Then
    \begin{align*}
        f(x\cdot_\phi y)=f\phi(x\cdot y)=\psi(f(x\cdot y))=\psi(f(x)\cdot'f(y))=f(x)\cdot'_\psi f(y),
    \end{align*}
    for all $x,y\in X$.
\end{proof}

\begin{thm} \label{thm:classification_p2_irretractable}
    Let $X$ be an irretractable cycle set of size $p^2$ where $p$ is a prime. Then there exists a unique $\Phi\in \mathcal{R}_p$ and $\alpha\in \Z_p^*$ satisfying ${}^\alpha \Phi=\Phi$ such that $X$ is isomorphic to the cycle set on $\Z_p\times \Z_p$ with multiplication
    $$(a,x)\cdot (b,y)= (\alpha b+\alpha x,\alpha y+\alpha \Phi(b-a)).$$
    If $\alpha= 1$, then the cycle sets are the ones that appear in \cref{thm:isoclasses_of_irretractable_cyclesets_sylow_case}. If $\alpha\neq 1$, then any automorphism of $(X,\cdot)$ is of the form $(a,x)\mapsto (\gamma a,\gamma x)$ for some $\gamma\in \Z_p^*$ with ${}^\gamma \Phi=\Phi$.
\end{thm}
\begin{proof}
    From the discussion preceding \cref{lem:deformation_of_cycle_set_by_automorphism} we know that $(X,\cdot)$ can be obtained by starting from a cycle set structure on $X$ whose permutation group is a $p$-group and deforming such cycle set by an automorphism of order coprime to $p$, in the sense of \cref{lem:deformation_of_cycle_set_by_automorphism}. 
    From \cref{thm:isoclasses_of_irretractable_cyclesets_sylow_case} it follows that, up to a cycle set isomorphism, $X=\Z_p\times \Z_p$ and 
    $$(a,x)\cdot (b,y)= (\alpha b+\alpha x+\beta,\alpha y+\alpha \Phi(b-a)),$$
    for some $\Phi\in \mathcal{R}_p$, $\alpha\in \Z_p^*$ and $\beta\in \Z_p$, satisfying ${}^\alpha \Phi=\Phi$.
    By \cref{lem:isomorphism_deformed} we may even assume $\beta=0$. We therefore get that up to isomorphism the multiplication on $X$ is precisely as in the statement. 
    
    Conversely, it follows directly from \cref{lem:deformation_of_cycle_set_by_automorphism} and \cref{lem:deformation_improved} that $\Z_p\times \Z_p$ with the given multiplication always yields an indecomposable irretractable cycle set. As a consequence of \cref{lem:isomorphism_deformed} we find that different choices of $\alpha$ and $\Phi$ yield non-isomorphic solutions and also that the automorphisms are the ones described in the statement.
\end{proof}

\begin{rem}
    Observe that if $\alpha\neq 1$, then
    $\pi(\G(X))\neq \pi(X)$.
    Hence the permutation braces 
    of these solutions are all examples of
    singular brace as defined in \cite{MR4644952}.
\end{rem}

\section{Summary}

We summarize our classification result in the following theorem.

\begin{thm} \label{thm:final_classification_cycle_sets}
Let $X$ be an indecomposable cycle set of size $p^2$. Then $X$ is isomorphic to a cycle set of one of the following forms:
\begin{enumerate}
    \item $X = \Z_{p^2}$, $x \cdot y = y + 1$,
    \item $X = \Z_p \times \Z_p$,
    \[
    (a,x) \cdot (b,y) = (b+1, y + \chi_0(b)S + \Phi(b-a)),
    \]
    where $\Phi: \Z_p \to \Z_p$ is a non-constant map with $\Phi(0) = 0$, $S \in \Z_p$ and $\chi_0: \Z_p \to \Z_p $ with
        \begin{align*}
            \chi_0(x) & = \begin{cases}
                1 & x = 0 \\
                0 & x \neq 0
            \end{cases}.
            \end{align*}
    The parameters $S,\Phi$ and $\Pri{S},\Pri{\Phi}$ define isomorphic cycle sets if and only if $S = \Pri{S}$ and $\alpha \Phi = \Pri{\Phi}$ for some $\alpha \in \Z_p^{\ast}$.
    \item $X = \Z_p \times \Z_p$,
    \[
    (a,x) \cdot (b,y) = (\alpha b + \alpha x , \alpha y + \alpha \Phi(b-a)),
    \]
    where $\Phi: \Z_p \to \Z_p$ is a non-constant map with $\Phi(x) = \Phi(-x)$ and $\alpha \in \Z_p^{\ast}$ is such that $\Phi(\alpha x) = \alpha \Phi(x)$.

    The parameters $\alpha, \Phi$ and $\Pri{\alpha}, \Pri{\Phi}$ define isomorphic cycle sets if and only if $\alpha = \Pri{\alpha}$ and there is a $\beta \in \Z_p^{\ast}$ such that $\beta^{-1} \Phi(\beta x) = \Pri{\Phi}(x)$, for all $x \in \Z_p$.
\end{enumerate}
These three cases are mutually exclusive.
\end{thm}
\begin{proof}
 \cref{cor:classification_p2_multipermutation} tells us that the indecomposable cycle sets of size $p^2$ that have finite multipermutation level $1$ and $2$, are exactly the ones described in (1) resp. (2). On the other hand, the irretractable cycle sets are classified, up to isomorphism in \cref{thm:classification_p2_irretractable} and make up case (3).
\end{proof}

Using the correspondence between cycle sets and set-theoretical solutions, the previous theorem, reformulated in terms of set-theoretical solutions, is the following.

\begin{cor}\label{cor:final_classification_solutions}
    Each non-degenerate, indecomposable, involutive set-theoretical solution $(X,r)$ to the Yang--Baxter equation of size $p^2$ for some prime $p$ is isomorphic to one of the following solutions:

    \begin{enumerate}
        \item $X = \Z_{p^2}$, with $r(x,y) = (y+1,x-1)$.
        \item $X = \Z_p \times \Z_p$, with
        \[
        r\begin{pmatrix}
            (a,x) \\ (b,y)
        \end{pmatrix} = \begin{pmatrix}
             ( b-1,y- \chi_0(b-1)S - \Phi(b-1-a))\\
             (a+1, x +  \chi_0(a)S + \Phi(a-b+1))
        \end{pmatrix}
        \]
        where $\Phi: \Z_p \to \Z_p$ is a non-constant map with $\Phi(0) = 0$, $S \in \Z_p$ and $\chi_0: \Z_p  \to \Z_p$ with
        \begin{align*}
            \chi_0(x) & = \begin{cases}
                1 & x = 0 \\
                0 & x \neq 0
            \end{cases}.
        \end{align*}
        The parameters $S,\Phi$ and $\Pri{S},\Pri{\Phi}$ define isomorphic solutions if and only if $S = \Pri{S}$ and $\alpha \Phi = \Pri{\Phi}$ for some $\alpha \in \Z_p^{\ast}$.
        
        \item $X = \Z_p \times \Z_p$, with
        \[
        r\begin{pmatrix}
            (a,x) \\ (b,y)
        \end{pmatrix} = \begin{pmatrix}
            (\alpha^{-1}b - x, \alpha^{-1} y - \Phi(\alpha^{-1}b  - x-a)) \\
            (\alpha a + y - \Phi(b - \alpha x - \alpha a), \alpha x + \Phi(\alpha a  - \alpha x - b)) 
        \end{pmatrix}
        \]
        where $\Phi: \Z_p \to \Z_p$ is a non-constant map with $\Phi(x) = \Phi(-x)$ and $\alpha \in \Z_p^{\ast}$ is such that $\Phi(\alpha x) = \alpha \Phi(x)$.
        
        The parameters $\alpha, \Phi$ and $\Pri{\alpha}, \Pri{\Phi}$ define isomorphic solutions if and only if $\alpha = \Pri{\alpha}$ and there is a $\beta \in \Z_p^{\ast}$ such that $\beta^{-1} \Phi(\beta x) = \Pri{\Phi}(x)$ for all $x \in \Z_p$.
        \end{enumerate}
\end{cor}

\subsection{Indecomposable set-theoretical solutions of size \texorpdfstring{$p^2$}{p2}}

    Recall that given a cycle set on $X$, the associated solution is given by 
\[
r_X(x,y)=(\sigma_x^{-1}(y),\sigma_x^{-1}(y)\cdot x).
\]
Therefore we can obtain all indecomposable solutions of size $p^2$
simply translating the cycle sets obtained in 
\cref{thm:final_classification_cycle_sets} to set-theoretical solutions.
In case \emph{(1)} we obtain solutions of the form $r(x,y)=(y-1,x+1)$,
as $\sigma_x:y\mapsto y+1$, thus $\lambda_x=\sigma_x^{-1}:y\mapsto y-1$
and $\rho_y:x\mapsto \sigma_x^{-1}(y)\cdot x=(y-1)\cdot x=x+1$.

In case \emph{(2)} we have that 
$\sigma_{(a,x)}:(b,y)\mapsto (b + 1, y + S \chi_0(b) + \Phi(b-a))$,
hence \[
\lambda_{(a,x)}(b,y)=\sigma_{(a,x)}^{-1}(b,y)=(b - 1, y - S \chi_0(b-1) - \Phi(b-1-a))
\]
and 
\begin{align*}
   \rho_{(b,y)}(a,x)&
   =\sigma_{(a,x)}^{-1}(b,y)\cdot (a,x)
   =(b - 1, y - S \chi_0(b-1) - \Phi(b-1-a))\cdot(a,x)\\
   &=(a+1, x+S\chi_0(a)+\Phi(a-b+1)).
\end{align*}
Thus the associated solution is 
\[
        r\begin{pmatrix}
            (a,x) \\ (b,y)
        \end{pmatrix} = \begin{pmatrix}
             ( b-1,y-S \chi_0(b-1) - \Phi(b-1-a))\\
             (a+1, x + S \chi_0(a) + \Phi(a-b+1)).
        \end{pmatrix}
\]
Finally in case \emph{(3)} we have
$\sigma_{(a,x)}:(b,y)\mapsto (\alpha b + \alpha x , \alpha y + \alpha \Phi(b-a))$,
hence, using that $\Phi(\alpha x)=\alpha\Phi(x)$ for all $x\in \Z_p$,
\begin{align*}
\lambda_{(a,x)}(b,y)&
=\sigma_{(a,x)}^{-1}(b,y)
=\left(\alpha^{-1}b-x , \alpha^{-1}y-\Phi\left(\alpha^{-1}b-x-a\right)\right)  
\end{align*}
and 
\begin{align*}
   \rho_{(b,y)}(a,x)&
   =\sigma_{(a,x)}^{-1}(b,y)\cdot (a,x)
   =\left(\alpha^{-1}b-x , \alpha^{-1}y-\Phi\left(\alpha^{-1}b-x-a\right)\right)\cdot(a,x)\\
   &=\left(\alpha a +\alpha \left(\alpha^{-1}(y-\Phi(b-\alpha x-\alpha a))\right), \alpha x+\alpha\Phi\left(a-\alpha^{-1}(b-\alpha x)\right)\right)\\
   &=\left(\alpha a +y-\Phi(b-\alpha x-\alpha a), \alpha x+\Phi(\alpha a-b+\alpha x)\right).
\end{align*}
Thus the associated solution is 
\[
        r\begin{pmatrix}
            (a,x) \\ (b,y)
        \end{pmatrix} = \begin{pmatrix}
            (\alpha^{-1}b -  x, \alpha^{-1} y - \Phi(\alpha^{-1} b - x - a)) )\\
            (\alpha a + y - \Phi(b - \alpha x - \alpha a), \alpha x + \Phi(\alpha a  - \alpha x - b)). 
        \end{pmatrix}
\]
In fact, these solutions are isomorphic to those constructed in \cite[Theorem 5.1]{CO_simple} as we will show in the remainder of this section. In particular, this answers \cite[Question 7.3]{CO_simple} affirmatively.
\begin{thm}[{\cite[Theorem 5.1]{CO_simple}}]
\label{COsimplesol}
    Let $p$ be a prime number. Let $t\in \Z_p$ be a non-zero element.
    Let $f:\Z_p\to \Z_p$ be a map such that
    \begin{enumerate}[label=\emph{(S\arabic*)}]
        \item\label{S1} $f(i) = f(-i)$, for all $i\in\Z_p$,
        \item\label{S2} $f(t^s i) = t^sf(i) - (t^s - 1)f(0)$, for all $i\in\Z_p$ and $s\in\Z$,
        \item\label{S3} $f$ is not a constant map.
    \end{enumerate}
    Let $X=\Z_p\times\Z_p$ and $r:X^2\to X^2$ be the map
    $r\begin{pmatrix}
      (i,j)\\
      (k,l)  
    \end{pmatrix}
    =\begin{pmatrix}
        \lambda_{(i,j)}(k,l)\\
        \lambda_{\lambda_{(i,j)}(k,l)}^{-1}(i,j)
    \end{pmatrix}$,
    where $\lambda_{(i,j)}(k,l)=(tk+j, t(l-f(tk+j-i)))$. 
    Then $(X,r)$ is a \emph{simple} solution of the YBE in the following sense: if $(Y,s)$ is a solution and $f: (X,r) \twoheadrightarrow (Y,s)$ is a surjective homomorphism, then $|Y| \in \{ 1 , |X|\}$.
\end{thm}

We will denote the solution associated with cycle sets of the form \emph{(3)} 
in \cref{thm:final_classification_cycle_sets} with parameters $\Phi,\alpha$ as 
$r^{\alpha,\Phi}$ with first component 
\[
\lambda^{\Phi,\alpha}_{(a,x)}:(b,y)\mapsto(\alpha^{-1}b - x, \alpha^{-1} y - \Phi( \alpha^{-1} b - x - a)) ).
\]
Similarly, we will denote the solution constructed in \cref{COsimplesol}
with parameters $f,t$ as $r^{f,t}$ with first component
\[
\ell^{f,t}_{(i,j)}:(k,l)\mapsto(tk+j, t(l-f(tk+j-i))).
\]
With this notation and fixing $X=\Z_p\times\Z_p$, it is easy to prove that the map 
$\Psi(i,j)=(i,-j)$
is an isomorphism of solutions
$\Psi:(X,r^{\Phi,\alpha})\to(X,r^{f_{\Phi,\alpha},\alpha^{-1}})$,
where $f_{\Phi,\alpha}:i\mapsto -\Phi(\alpha i)$, since
\begin{align*}
    \lambda^{\Phi,\alpha}_{\Psi(i,j)}(\Psi(k,l))&
    =\lambda^{\Phi,\alpha}_{(i,-j)}((k,-l))
    =(\alpha^{-1}(k + \alpha j), \alpha^{-1} (-l - \Phi(k + \alpha j - \alpha i)) )\\
    &=(\alpha^{-1}k + j, \alpha^{-1} (-l + f_{\Phi,\alpha}(\alpha^{-1}k + j - i)) )\\
    &=\Psi(\alpha^{-1}k + j, \alpha^{-1} (l - f_{\Psi,\alpha}(k + \alpha j - \alpha i))\\
    &=\Psi\left(\ell^{f_{\Phi,\alpha},\alpha^{-1}}_{(i,j)}(k,l)\right).
\end{align*}
It remains to show that,
with the conditions for $\Phi$ and $\alpha$
given in \cref{thm:final_classification_cycle_sets},
the parameters $f=f_{\Phi,\alpha}$ and $t=\alpha^{-1}$
satisfy the properties required by \cref{COsimplesol}.
Since $\Phi$ satisfies \ref{S1} and \ref{S3}, so does $f_{\Phi,\alpha}$.
Moreover, since $\Phi(\alpha i)=\alpha\Phi(i)$, we have that
$f_{\Phi,\alpha}(\alpha^{-s}i)=-\Phi(\alpha \alpha^{-s}i)=-\alpha^{-s}\Phi(\alpha i)=\alpha^{-s}f_{\Phi,\alpha}(i)$.
Hence $f_{\Phi,\alpha}$ satisfies \ref{S2} if and only if
$(\alpha^{-s}-1)f_{\Phi,\alpha}(0)=0$ for all $s\in\Z$,
which is equivalent to $(\alpha-1)\Phi(0)=0$.
But the latter is a consequence of the properties of $\Phi$ and $\alpha$ as 
$\Phi(0)=\Phi(\alpha 0)=\alpha\Phi(0)$.

\subsection{Enumeration of indecomposable, irretractable cycle sets of size \texorpdfstring{$p^2$}{p2}} 

In this subsection, we will use the following convention: for a group $G$ acting on a set $X$ by an action $(g,x) \mapsto {}^gx$, we denote the set of fixed points by 
$$\Fix_X(G) = \{ x \in X \mid {}^gx = x \text{ for all } g \in G \}.$$
Recall that $\mathcal{F}_p$ has been defined as the set of all non-constant maps $\Phi: \Z_p \to \Z_p$ such that $\Phi(-A) = \Phi(A)$.

By \cref{thm:final_classification_cycle_sets}, every irretractable cycle sets can be described by a pair $(\Phi,\alpha)$ where $\alpha \in \Z_p^{\ast}$ satisfies ${}^{\alpha}\Phi = \Phi$ for all $A \in \Z_p$,
where ${}^{\alpha} \Phi:A\mapsto\alpha^{-1} \Phi(\alpha A)$.
Note that this is an action of $\Z_p^{\ast} \cong \Z_{p-1}$ on $\mathcal{F}_p$. Furthermore, $(\Phi,\alpha)$ and $(\Phi^{\prime},\alpha^{\prime})$ define isomorphic cycle sets, if and only if $\alpha = \alpha^{\prime}$ and $\Phi^{\prime} = {}^{\beta}\Phi$ for some $\beta \in \Z_p^{\ast}$. It follows directly that for $p=2$ we find 2 non-isomorphic indecomposable irretractable cycle sets of size $4$.

Assume $p \neq 2$ from now on. For a pair $(\Phi,\alpha)$ with $\Phi(0) \neq 0$, the condition $\Phi(\alpha A) = \alpha \Phi(A)$ forces $\alpha = 1$ and $\Phi$ has to be non-constant and even. As $\Phi(0) \neq 0$, there is exactly one $\beta \in \Z_p^{\ast}$ such that ${}^{\beta}\Phi(0) = 1$. Therefore, each cycle set with parameters $(\Phi,\alpha)$, $\Phi(0) \neq 0$, is isomorphic to a unique cycle set with parameters $(\Tilde{\Phi},1)$ where $\Tilde{\Phi}(0) =1$, and these parameters define mutually non-isomorphic cycle sets. We therefore find $n_p=p^{\frac{p-1}{2}}-1$ such cycle sets. 

Now assume $\Phi(0) = 0$. The condition $p \neq 2$ excludes $\alpha = -1$, as $\Phi(A) = \Phi(-A)$. More general, $-1 \not\in \left\langle \alpha \right\rangle \leq \Z_p^{\ast}$. Writing $p-1 = 2^kl$, with $2\nmid l$ we see that 
$$\alpha \in \{ x \in \Z_p^{\ast} : x^l = 1 \} =: \zeta_l \cong \Z_l.$$

As $\zeta_l$ is cyclic, there are $\varphi(d)$ elements $\alpha \in \zeta_l$ with multiplicative order $\ord(\alpha) = d$, where $\varphi$ denotes the Euler $\varphi$-function.
    
In order to count the number of orbits under the action of $\Z_p^{\ast}$, we use Burnside's lemma. Note that only the elements of $\zeta_l$ have fixed points in $\mathcal{F}_p$, so we can restrict to those. Given $\alpha \in \zeta_l$ with $\ord(\alpha) = d$, a function $\Phi \in \mathcal{F}_p$ that satisfies ${}^{\alpha}\Phi = \Phi$, is already defined by its values on coset representatives of $\Z_p^{\ast}/\left\langle -1, \alpha \right\rangle$. Under the restriction that $\Phi$ is non-constant and $\Phi(0) = 0$, there are $p^{\frac{p-1}{2d}}-1$ choices for $\Phi$.

As $\left\langle \alpha \right\rangle$ leaves each of these options invariant, we only need to count orbits with respect to the induced action of $\Z_p^{\ast}/\left\langle \alpha \right\rangle$ on $\Fix_{\mathcal{F}_p}(\langle\alpha\rangle)$. Given an element $[\beta] \in \Z_p^{\ast}/\left\langle \alpha \right\rangle$, it has fixed points in $\Fix_{\mathcal{F}_p}(\langle\alpha\rangle)$ if and only if $\beta \in \zeta_l$. Assuming the latter, and letting $c = \ord_{\Z_p^{\ast}}(\beta)$, we get $p^{\frac{p-1}{2c}}-1$ elements in $\Fix_{\mathcal{F}_p}(\langle[\beta]\rangle)$.

Using Burnside's lemma and the fact that there are $\varphi(\frac{c}{d})$ elements $[\beta] \in \zeta_l/\left\langle \alpha \right\rangle$ with $\ord_{\Z_p^{\ast}}(\beta) = c$, we get that there are
\[
    \frac{d}{p-1} \sum_{c ; d|c|l} \varphi\left(\frac{c}{d}\right) (p^{\frac{p-1}{2c}}-1)
\]
equivalence classes for parameters of the form $(\Phi,\alpha)$ where $\ord(\alpha) = d$. Considering that there are $\varphi(d)$ such $\alpha \in \zeta_l$, we get the following number of non-isomorphic cycle sets with parameters $(\Phi,\alpha)$, $\Phi(0) \neq 0$:
\[
    n^{\prime}_p = \frac{1}{p-1} \sum_{c,d ; d | c | l} d\varphi(d) \varphi\left(\frac{c}{d}\right) (p^{\frac{p-1}{2c}}-1) = \frac{1}{p-1} \sum_{c,d ; d | c | l} d\varphi(d) \varphi\left(\frac{c}{d}\right) (p^{2^{k-1}\frac{l}{c}}-1).
\]
Note that the function
\[
    \psi(n) = \sum_{d|n} d\varphi(d) \varphi \left( \frac{n}{d} \right)
\]
is a convolution of multiplicative functions. Here, \emph{multiplicative} means $\mu(mn) = \mu(m)\mu(n)$ for coprime positive integers $m,n$. So also $\psi$ is a multiplicative function which evaluates on prime powers $q^{\nu}$, $\nu \geq 1$, as 
\begin{align*}
        \psi(q^{\nu}) & = \sum_{k=0}^{\nu} q^k\varphi(q^k) \varphi(q^{\nu -k}) \\
        & = q^{\nu}(q-1) q^{\nu -1} + (q-1) q^{\nu -1} + \sum_{k=1}^{\nu -1}(q-1)^2 q^{\nu + k - 2}\\
        & = (q-1) \left(q^{2\nu -1} + q^{\nu - 1} + (q-1)q^{\nu -1} \sum_{k=1}^{\nu-1}q^{k-1} \right) \\
        & = (q-1) \left( q^{2\nu -1} + q^{\nu - 1} + q^{\nu -1}(q^{\nu -1} -1) \right) \\
        & = (q-1) \left( q^{2\nu -1} + q^{2\nu -2} \right) \\
        & = (q^2 -1) q^{2\nu -2}.
\end{align*}
For a number with prime factorization $n = \prod_i q_i^{\nu_i}$, we therefore get
\[
    \psi(n) = \prod_i (q_i^2 -1) q_i^{2\nu_i -2}.
\]
The total number of indecomposable, non-isomorphic, irretractable cycle sets of size $p^2$ can therefore be described as:
\[
    n_p+n_p'  = p^{\frac{p-1}{2}} -1 + \sum_{d|l} \psi\left(\frac{l}{d}\right)\frac{p^{2^{k-1}d}-1}{p-1}
\]
where $p-1 = 2^kl$ with $2\nmid l$.

\section*{Acknowledgements}
This work is partially supported by the project OZR3762 of Vrije Universiteit Brussel and by Fonds Wetenschappelijk Onderzoek - Vlaanderen, via the Senior Research Project G004124N.

The first author expresses his gratitude to the Alexander Humboldt Foundation which funds, by means of a Feodor Lynen fellowship, the research project that encompasses the research conducted for this article.

The second author is supported by Fonds Wetenschappelijk Onderzoek - Vlaanderen, via a PhD Fellowship fundamental research, grant 11PIO24N. 

The third author is supported by Fonds Wetenschappelijk Onderzoek - Vlaanderen, via a PhD Fellowship fundamental research,
grant 1160524N.

The authors would also like to express their gratitude to Marco Castelli for pointing out that one results in an earlier version of the manuscript was already proved by Rump.

\bibliography{references}

\begin{thebibliography}{10}

\bibitem{AMV_Cyclesets}
O.~Akgun, M.~Mereb, and L.~Vendramin.
\newblock Enumeration of set-theoretic solutions to the {Y}ang-{B}axter equation.
\newblock {\em Math. Comp.}, 91:1469--1481, 2020.

\bibitem{BCJO_irretrsqfree}
D.~Bachiller, F.~Ced\'{o}, E.~Jespers, and J.~Okni\'{n}ski.
\newblock A family of irretractable square-free solutions of the {Y}ang-{B}axter equation.
\newblock {\em Forum Math.}, 29(6):1291--1306, 2017.

\bibitem{Baxter_YB}
R.~J. Baxter.
\newblock Partition function of the eight-vertex lattice model.
\newblock {\em Ann. Physics}, 70(1):193--228, 1972.

\bibitem{castelli2023studying}
M.~Castelli and S.~Trappeniers.
\newblock Studying solutions of the {Y}ang-{B}axter equation through skew braces, with an application to indecomposable involutive solutions with abelian permutation group.
\newblock {\em arXiv.2303.00581}, 2023.

\bibitem{CJO_Braces}
F.~Ced\'{o}, E.~Jespers, and J.~Okni\'{n}ski.
\newblock Braces and the {Y}ang-{B}axter equation.
\newblock {\em Comm. Math. Phys.}, 327(1):101--116, 2014.

\bibitem{CO_simple}
F.~Ced\'{o} and J.~Okni\'{n}ski.
\newblock Constructing finite simple solutions of the {Y}ang-{B}axter equation.
\newblock {\em Adv. Math.}, 391:Paper No. 107968, 40, 2021.

\bibitem{CO_SquarefreeIndecomposable}
F.~Ced\'{o} and J.~Okni\'{n}ski.
\newblock Indecomposable solutions of the {Y}ang-{B}axter equation of square-free cardinality.
\newblock {\em Adv. Math.}, 430:Paper No. 109221, 26, 2023.

\bibitem{Drinfeld_Problems}
V.~G. Drinfeld.
\newblock On some unsolved problems in quantum group theory.
\newblock In P.~P. Kulish, editor, {\em Quantum Groups}, pages 1--8, Berlin, Heidelberg, 1992. Springer Berlin Heidelberg.

\bibitem{ESS_YangBaxter}
P.~Etingof, T.~Schedler, and A.~Soloviev.
\newblock {Set-theoretical solutions to the quantum {Y}ang-{B}axter equation}.
\newblock {\em Duke Math. J.}, 100(2):169 -- 209, 1999.

\bibitem{MR4716443}
E.~Feingesicht.
\newblock Dehornoy's class and {S}ylows for set-theoretical solutions of the {Y}ang-{B}axter equation.
\newblock {\em Internat. J. Algebra Comput.}, 34(1):147--173, 2024.

\bibitem{GIVdB_IType}
T.~Gateva-Ivanova and M.~Van~den Bergh.
\newblock Semigroups of {I}-type.
\newblock {\em J. Algebra}, 206:97--112, 1998.

\bibitem{huppert1983endliche}
B.~Huppert.
\newblock {\em Endliche Gruppen}.
\newblock Finite groups / B. Huppert, N. Blackburn. Springer, 1983.

\bibitem{jedlicka_pilitowska}
P.~Jedlička and A.~Pilitowska.
\newblock Indecomposable involutive solutions of the {Y}ang-{B}axter equation of multipermutation level 2 with non-abelian permutation group.
\newblock {\em Journal of Combinatorial Theory, Series A}, 197:105753, 2023.

\bibitem{MR4717099}
V.~Lebed, S.~Ram\'{\i}rez, and L.~Vendramin.
\newblock Involutive {Y}ang-{B}axter: cabling, decomposability, and {D}ehornoy class.
\newblock {\em Rev. Mat. Iberoam.}, 40(2):623--635, 2024.

\bibitem{Lucchini_imprimitive}
A.~Lucchini.
\newblock On imprimitive groups with small degree.
\newblock {\em Rend. Semin. Mat. Univ. Padova}, 86:131--142, 1991.

\bibitem{Rump_Decomposition}
W.~Rump.
\newblock A decomposition theorem for square-free unitary solutions of the quantum {Y}ang–{B}axter equation.
\newblock {\em Adv. Math.}, 193(1):40--55, 2005.

\bibitem{Rump_braces}
W.~Rump.
\newblock Braces, radical rings, and the quantum {Y}ang–{B}axter equation.
\newblock {\em J. Algebra}, 307(1):153--170, 2007.

\bibitem{MR4644952}
W.~Rump.
\newblock Primes in coverings of indecomposable involutive set-theoretic solutions to the {Y}ang-{B}axter equation.
\newblock {\em Bull. Belg. Math. Soc. Simon Stevin}, 30(2):260--280, 2023.

\bibitem{SV_OnSkewBraces}
A.~Smoktunowicz and L.~Vendramin.
\newblock On skew braces (with an appendix by {N}. {B}yott and {L}. {V}endramin).
\newblock {\em J. Comb. Algebra}, 2(1):47--86, 2018.

\bibitem{Turaev_links}
V.~G. Turaev.
\newblock The {Y}ang-{B}axter equation and invariants of links.
\newblock {\em Invent. Math.}, 92:527--553, 1988.

\bibitem{Yang_YB}
C.~N. Yang.
\newblock Some exact results for the many-body problem in one dimension with repulsive delta-function interaction.
\newblock {\em Phys. Rev. Lett.}, 19:1312--1315, Dec 1967.

\end{thebibliography}
\bibliographystyle{abbrv}

\end{document}